\documentclass[11pt,a4paper]{article}

\usepackage[utf8]{inputenc}
\usepackage[T1]{fontenc}
\usepackage{amsmath,amssymb,amsthm,mathtools,mathrsfs}
\usepackage[margin=2.8cm]{geometry}
\usepackage{xcolor}
\usepackage[colorlinks=true,linkcolor=blue!55!black,citecolor=green!45!black,urlcolor=blue!55!black]{hyperref}
\usepackage{enumitem}
\usepackage{comment}

\newif\ifshowchanges
\showchangestrue         

\theoremstyle{plain}
\newtheorem{theorem}{Theorem}[section]
\newtheorem{proposition}[theorem]{Proposition}
\newtheorem{corollary}[theorem]{Corollary}
\newtheorem{lemma}[theorem]{Lemma}
\theoremstyle{remark}
\newtheorem{remark}[theorem]{Remark}
\newtheorem{example}[theorem]{Example}
\theoremstyle{definition}
\newtheorem{definition}[theorem]{Definition}
\newtheorem{assumption}[theorem]{Assumption}

\newcommand{\R}{\mathbb R}
\newcommand{\Z}{\mathbb Z}
\newcommand{\N}{\mathbb N}
\newcommand{\T}{\mathbb T}
\newcommand{\E}{\mathbb E}
\newcommand{\PP}{\mathbb P}
\newcommand{\Var}{\operatorname{Var}}
\newcommand{\Cov}{\operatorname{Cov}}
\newcommand{\cum}{\operatorname{cum}}
\newcommand{\dd}{\,\mathrm d}
\newcommand{\la}{\langle}
\newcommand{\ra}{\rangle}
\newcommand{\XN}{\mathcal X_N^{P}}
\newcommand{\YN}{\mathcal Y_N^{P}}
\newcommand{\WN}{\mathcal W_N}
\newcommand{\TN}{\mathbb T^d_N}
\newcommand{\LN}{\Lambda_N}

\DeclareMathOperator{\supp}{supp}
\DeclareMathOperator{\mst}{mst}

\begin{document}

\title{Gaussian fluctuations of differential observables\\ of stationary lattice fields}
\author{%
  Fabio Coppini\footnote{Mathematical Institute, Utrecht University, Budapestlaan 6, 3584 CD Utrecht, The Netherlands.
    {f.coppini@uu.nl}}
  \and 
  Wioletta M. Ruszel\footnote{Mathematical Institute, Utrecht University, Budapestlaan 6, 3584 CD Utrecht, The Netherlands.
  {w.m.ruszel@uu.nl}}
  }
\date{\today}
\maketitle

\begin{abstract}
We study Gaussian fluctuations of differential observables of centered stationary random fields on the discrete $d$-dimensional torus of mesh $1/N$. Our main result is a Central Limit Theorem for the fields $f\mapsto N^{-d/2-m}\sum_x\phi_x^N(Pf)(x/N)$, where $P$ is a scalar homogeneous constant-coefficient differential operator of order $m\geq0$ with Fourier symbol $p$, e.g., $P=\Delta$ and $m=2$. Under suitable assumptions on the spectral scaling limit $\gamma$ of the covariance and the normalized cumulants of order $r\geq3$, we prove convergence in law in a negative Sobolev space to a centered Gaussian random distribution with covariance multiplier $|p(k)|^2\gamma(k)$. The two factors separate the effects of the operator and the field and, notably, differentiation can compensate a low-frequency spectral singularity. For the lattice Gaussian Free Field, this gives a projected white-noise limit for the gradient. The proofs use Fourier analysis and the method of moments and cumulants, together with a tightness argument. Finally, we give two sufficient criteria for the cumulant assumption, namely a tree--graph bound and the Dobrushin uniqueness condition for finite-range Gibbs fields, and recover the white-noise limits of the centered Ising magnetization and Potts colour fluctuation fields at high temperature.

\end{abstract}

\section{Introduction}\label{sec:intro}
\subsection{Motivation and main results}\label{sec:motivation}

The central limit theorems (CLT) for fluctuations of stationary lattice fields in the literature
typically rest on a structural hypothesis about the field: most often the FKG
inequalities together with finite susceptibility \cite{newman_normal_1980}, a martingale or
mixing structure \cite{kipnis_central_1986,bolthausen_central_1982}, or an explicit Gaussian/chaos representation \cite{coppini_central_2025}.  
The starting point of this work is to formulate conditions directly in terms of the covariance and the cumulants, without assuming any of these additional structures. More precisely, we combine a spectral scaling assumption with quantitative smallness of the higher cumulants to obtain a macroscopic Gaussian fluctuation limit.

Let $d\geq1$ and let $\phi^N=(\phi_x^N)_{x\in\Lambda_N}$ be a centered stationary random field on the discrete torus, where $\Lambda_N=\{0,\ldots,N-1\}^d$. We identify $x\in\Lambda_N$ with $x/N\in\mathbb T^d$. Let $P=\sum_{|\alpha|=m}c_\alpha\partial^\alpha$ be a scalar homogeneous differential operator of exact order $m\geq0$, with constant real coefficients. For every real $f\in C^\infty(\mathbb T^d)$, define
\begin{equation}\label{eq:def:PN}
\mathcal X_N^P(f):=N^{-d/2-m}\sum_{x\in\Lambda_N}\phi_x^N(Pf)(x/N).
\end{equation}
For instance, $P=\operatorname{Id}$ gives the usual volume average, while $P=\partial_v:=v\cdot\nabla$ or $P=\Delta$ allows us to study differential observables of the field. We are interested in conditions on the covariance and the cumulants of $\phi^N$ such that $\mathcal X_N^P$ satisfies a Central Limit Theorem (CLT). The field need not be Gaussian, and we do not assume the FKG inequalities or a martingale representation.

Denote the covariance by $\Gamma_N(z)=\operatorname{Cov}(\phi_0^N,\phi_z^N)$ and its discrete Fourier transform by $\widehat\Gamma_N$. Let $p$ be the Fourier symbol of $P$, so that $\widehat{Pf}(k)=p(k)\widehat f(k)$. Our main results can be stated informally as follows.

\begin{theorem}[Informal version of Theorem~\ref{thm:main}]\label{thm:intro-informal}
Let $\phi^N$ have finite moments of all orders. Assume the spectral scaling and higher-cumulant conditions of Assumptions \ref{ass:spectral}  and \ref{ass:cumulants}, and denote the limiting covariance profile by $\gamma$. For the rescaled covariance profile, we obtain
\[
\mathcal X_N^P\overset{d}\longrightarrow\mathcal X^P
\qquad \text{ in} \text{  } H^{-s}(\mathbb T^d)
\text{ for } s\text{ sufficiently large},
\]
where $\mathcal X^P$ is a centered Gaussian random distribution with covariance 
\[
\mathbb E[\mathcal X^P(f)\mathcal X^P(g)]
= \displaystyle\sum_{k\in\mathbb Z^d}
\overline{\widehat f(k)}\widehat g(k)|p(k)|^2\gamma(k),
\]
and $f,g$ are suitable test functions.
\end{theorem}
The limiting covariance is determined by the product \(|p|^2\gamma\). When this product is constant, the limit is white noise projected onto the corresponding Fourier modes; see Corollary \ref{cor:white-noise}. A whole-space version of the previous theorem is formulated in Theorem~\ref{thm:wholespace}. We provide useful criteria to  check the cumulant assumption, for example through a tree--graph bound, see Equation \eqref{def:treegraph}, and recover well-known results in the literature such as the high-temperature Ising magnetization and Potts colour-field limits, without a separate cumulant computation for each model. Tree-graph bound estimates originate, for instance, through weighted dependency graphs \cite{feray_weighted_2018,dousse_weighted_2019}.
\begin{proposition}[Informal version of Proposition~\ref{prop:tree-graph}]
Suppose the fields satisfy the tree--graph bound (see \eqref{def:treegraph}). Then, for every $r\geq 3$ and fixed smooth test functions $f_1,\ldots, f_r$,
\[
\bigl|\operatorname{cum}(\mathcal X_N^P(f_1),\ldots,\mathcal X_N^P(f_r))\bigr|
\lesssim N^{d-r(d/2+m)}\longrightarrow0 \qquad \text{ as }N\to\infty.
\]
In particular, the higher-cumulant assumption of Theorem 1.1 is satisfied. 
\end{proposition}
The main theorem conclusion holds once the spectral scaling assumption has also been verified. A straightforward application of it applies to statistical mechanics models in the Dobrushin uniqueness regime, see Equation \eqref{eq:dobrushin-condition-new}. For the translation-invariant Gibbs distributions of spin models, we show that the Dobrushin condition implies a tree–graph bound for the cumulants of bounded single-site observables.

\begin{corollary}[Informal version of Corollary~\ref{cor:dobrushin-fluctuation-fields}]
Let $(\sigma_x)_x$ denote a spin configuration and let $F^1,\ldots,F^n$ be bounded single-site observables. For real $f,g\in C_c^\infty(\mathbb R^d)$, call the centered observable 
\[
\phi_x^a=F^a(\sigma_x)-\mathbb E_\mu[F^a(\sigma_x)],
\qquad
\mathcal W_N^a(f)=N^{-d/2}\sum_{x\in\mathbb Z^d}\phi_x^a f(x/N).
\]
Then the Dobrushin condition for the spin model implies that
\[
(\mathcal W_N^1,\ldots,\mathcal W_N^n)
\xrightarrow{\mathrm d}\mathcal W_\chi
\quad\text{in }H^{-s}_{\mathrm{loc}}(\mathbb R^d;\mathbb R^n),\quad s>d/2,
\]
where $\mu$ is the unique Gibbs state and $\mathcal W_\chi$ is centered Gaussian white noise with
\[
\mathbb E[\mathcal W_\chi^a(f)\mathcal W_\chi^b(g)]
=\chi_{ab}\int_{\mathbb R^d}f(x)g(x)\,dx,
\qquad
\chi_{ab}=\sum_{z\in\mathbb Z^d}\operatorname{Cov}_\mu(\phi_0^a,\phi_z^b).
\]
The susceptibility matrix is finite and positive semidefinite; the limit may be degenerate.
\end{corollary}

\medskip

The theory covers two regimes: finite susceptibility, where the differentiated microscopic observable carries the limit, and infinite susceptibility, where an operator of order $m>0$ compensates the infrared singularity of the spectral density and the macroscopic observable \eqref{eq:def:PN} is the nondegenerate object. In this last case,  Proposition~\ref{prop:trichotomy} makes precise, at the level of the covariance, the observation that differentiation removes
correlation. For a field whose spectral density behaves as $c_0|\theta|^{-\nu}$ near the origin, with $0\leq\nu<d$, we distinguish
three regimes according to the order $m$ of the differential operator.
Under the normalization $N^{-d/2-m}$, the variance vanishes when
$2m>\nu$, converges to a finite nonzero limit when $2m=\nu$, and grows
like $N^{\nu-2m}$ when $2m<\nu$, provided $Pf\not\equiv0$.
Notably, in the matching case the limiting covariance has multiplier
$c_0|p(\xi)|^2|\xi|^{-2m}$, which is bounded but need not be constant. This identifies when the differential order balances the spectral
singularity and the normalization of our main theorem gives nontrivial
fluctuations. When differentiation is insufficient, the alternative
normalization $N^{-(d+\nu)/2}$ restores a finite covariance limit;
for Gaussian fields, the resulting limit is a differentiated fractional
Gaussian field. We stress that this balance does not by itself imply
Gaussianity: in the matching regime, the higher-cumulant condition
allows us to apply the whole-space CLT of
Theorem~\ref{thm:wholespace}, whereas non-central limits may arise
without it; see Remark~\ref{rem:noncentral}.

\medskip

The novel contribution is threefold.  First, the limit is established at the level of the random
field rather than for a single fixed
linear statistic, i.e., we provide a distribution-valued central limit theorem.  Second, it isolates a unifying mechanism: the limiting covariance
factorizes in Fourier space as $|p(k)|^2\gamma(k)$, separating the deterministic effect of the operator $P$ from the probabilistic infrared profile $\gamma$ of the field, so that a single theorem produces white noise and gradient free fields according to how these two factors combine.  Third, it makes quantitative the heuristic that differentiation removes correlation, each derivative multiplies the spectral density by $|k|^2$ and so suppresses the low-frequency mass responsible for long-range correlation, recovering the white-noise limit for the gradient of the Gaussian free field
\cite{newman_gaussian_2018} as a special case (the same heuristic underlies the decorrelation
of the squared gradient observed in \cite{cipriani_properties_2023}). The method is deliberately elementary. Once the covariance scaling
has been identified and the spectral bounds verified, the remaining
model-specific input is the higher-cumulant hypothesis, see
Assumption~\ref{ass:cumulants}. For many models, suitable cumulant
bounds are already available through weighted dependency graphs
\cite{feray_weighted_2018,dousse_weighted_2019}, cluster expansions
\cite{kotecky_cluster_1986}, or the Breuer--Major diagram estimate
\cite{breuer_central_1983}. Once these inputs are available, the
distribution-valued and differentiated forms of the CLT follow
without FKG inequalities, Malliavin calculus, or a martingale
approximation.

\subsection{Related literature}\label{sec:literature}

It is known that the limiting variance of a fluctuation observable should be expressible through the spectral density of the covariance, and governed by its behaviour near zero frequency: Bierm\'e, Bonami and Le\'on \cite{bierme_central_2011} prove central limit theorems for quadratic variations of Gaussian fields with the limiting variance stated directly in terms of the spectral density and its asymptotically self-similar low-frequency profile; the Tauberian and Abelian theorems of Leonenko and Olenko \cite{leonenko_tauberian_2013} (see also the monograph \cite{leonenko_limit_1999}) make precise the correspondence between covariance decay at infinity and the spectral-density singularity at the origin of the type entering Section~\ref{sec:paradigm}. Moreover, the classical non-central limit theorems \cite{dobrushin_non-central_1979,taqqu_convergence_1979} are constructed from this precise low-frequency singularity.  Our contribution within this lineage is twofold: on the one hand, the factorized multiplier $|p(k)|^2\gamma(k)$ shows how the differential operator $P$ enters as the deterministic Fourier symbol $|p|^2$ multiplying the probabilistic infrared profile $\gamma$, the low-frequency scaling limit of the spectral density. On the other hand, to the authors' knowledge, this is also the first time where the spectral density of the covariance is applied to random fields on the lattice. In a complementary direction, when the spectral measure is not absolutely continuous the picture changes qualitatively: Gass, Peccati and Stecconi \cite{gass_scars_2026} show that  observables of Berry's singular-spectrum random-wave model fall instead into a fractional-Gaussian-field universality class, on the far side of the absolutely-continuous-spectrum hypothesis used here.

The volume average ($P=\mathrm{Id}$) of a stationary field is a classical observable of
statistical mechanics.  Newman \cite{newman_normal_1980} proved that, for stationary fields
satisfying the FKG inequalities, finite susceptibility implies a central limit theorem for
block averages; our results are a field-level (distribution-valued) counterpart valid for
differential observables, at the cost of a cumulant hypothesis.  Newman and Wu
\cite{newman_gaussian_2018} proved white-noise limits for the fluctuation field of the
classical XY model and of a class of gradient models as the temperature tends to zero with
the mesh; Example~\ref{ex:gff-gradient} recovers the exactly Gaussian core of this phenomenon.  For
$\nabla\varphi$-interface models with uniformly convex potential, the scaling limit is a
(homogenized) continuum Gaussian free field
\cite{naddaf_homogenization_1997,giacomin_equilibrium_2001,miller_fluctuations_2011,armstrong_c2_2022,armstrong_quantitative_2024},
with extensions to non-convex potentials
\cite{biskup_scaling_2011,cotar_existence_2012,buchholz_phase_2021}, local limit theorems
\cite{wu_local_2025} and strong central limit theorems \cite{conlon_strong_2013},
squares of gradient fields
\cite{deuschel_ray-knight_2024,cipriani_properties_2023}, and membrane models
\cite{thoma_thermodynamic_2023}.

For non-linear functionals of Gaussian fields, central limit theorems go back to
Breuer--Major \cite{breuer_central_1983} and Giraitis--Surgailis \cite{giraitis_clt_1985},
with non-central limits in the long-memory regime due to Dobrushin--Major
\cite{dobrushin_non-central_1979} and Taqqu \cite{taqqu_convergence_1979}; modern proofs
use Wiener chaos and the fourth-moment theorem
\cite{nourdin_normal_2012,peccati_wiener_2011,janson_gaussian_1997}, see
\cite{coppini_central_2025,coppini_wick_2025} for the lattice-field formulation and
\cite{mcauley_limit_2026} for recent limit theorems for non-local functionals.  Hermite
polynomial functionals of such fields satisfy our cumulant hypothesis under the Breuer--Major
summability condition, through the diagram estimate, but not the tree--graph bound of
\S\ref{sec:treegraph} with the covariance weight (Remark~\ref{rem:nonlinear-gaussian});
we do not pursue them here, our emphasis being a distribution-valued limit for
differential observables of not necessarily Gaussian stationary fields.  The underlying
tool---the method of moments and cumulants for limit theorems of random fields---is
classical \cite{peccati_wiener_2011,janson_gaussian_1997}; what is new here is its
field-level formulation for the differential observable \eqref{eq:def:PN} and the
resulting factorized covariance $|p|^2\gamma$.

A further point, which we regard as one of the main ones, is quantitative: under the
tree--graph criterion the cumulants of the fluctuation field obey the explicit all-orders
decay $|\cum(\XN(f_1),\dots,\XN(f_r))|\lesssim N^{\,d-r(d/2+m)}$ (see Equation \eqref{eq:cum-linear}
below), valid for every differential order $m\ge0$; under factorial control of the constants it yields, for the order-zero observable, a Berry--Esseen rate for each fixed test function, we refer to Remark~\ref{rem:berry-esseen}.  An explicit cumulant decay rate at this level of generality---for
differential observables of arbitrary order, at the level of the field---does not seem to have been isolated before. Finally, for the Ising model, cumulant bounds of tree--graph type follow from the weighted dependency graphs of Dousse--F\'eray
\cite{dousse_weighted_2019}, which reach beyond the Dobrushin regime used in \S\ref{sec:dobrushin-tree-graph}; see \cite{aizenman_sharpness_1987,duminil-copin_new_2016} for the sharpness of the
phase transition.

When cumulants are unavailable or do not vanish, Gaussianity may still hold through routes
outside the present framework: martingale approximations, or the FKG-based criterion of Newman
\cite{newman_normal_1980} for associated fields of finite susceptibility.  These are genuinely
distinct mechanisms, and in the \emph{long-memory} regime---by which we mean that the
spectral density is singular at the origin; this implies in particular that the covariance is
not absolutely summable and the susceptibility diverges---they may fail altogether, the limit being non-central, see Section~\ref{sec:paradigm} for more on this.

\subsection{Organization}\label{sec:organization}
The rest of the paper is organized as follows. Section~\ref{sec:setting} gives the notation and the full statements of the main results. Section~\ref{sec:proofs} contains the proofs. Section~\ref{sec:duality} proves the duality between macroscopic differential and microscopic difference operators. Section~\ref{sec:applications} gives the tree--graph and Dobrushin criteria and their applications. Finally, Section~\ref{sec:paradigm} discusses differentiation in the long-memory regime.

\section{Setting and main results}\label{sec:setting}

\subsection{Notation and conventions}\label{sec:notation}
Throughout, $d\ge1$.  We work on the continuous torus $\T^d:=\R^d/\Z^d$ and approximate it
by the discrete torus
\[
\TN:=(\Z/N\Z)^d,\qquad
\LN:=\{0,1,\dots,N-1\}^d,
\]
so that $|\LN|=N^d$ and the embedding $x\mapsto x/N$ maps $\LN$ onto the mesh-$1/N$
grid of $\T^d$.
We identify the dual group of $\TN$ with $\widehat\TN:=(-N/2,N/2]^d\cap\Z^d$, 
and for $k\in\Z^d$ write $\la k\ra_N\in\widehat\TN$ for the representative of $k\bmod N$, so that
$|\la k\ra_N|\le|k|$ and, as a sufficient condition,
$\la k\ra_N=k$ whenever $|k|\le N/2$. All lattice difference operators are taken with periodic boundary conditions on $\TN$, hence
summation by parts is exact:
\begin{equation}\label{eq:exact-sbp}
\sum_{x\in\LN}(\nabla_{N,j}u)_x\,v_x=-\sum_{x\in\LN}u_x\,(\nabla_{N,j}^*v)_x,
\qquad
(\nabla_{N,j}^*v)_x:=v_{x}-v_{x-e_j}.
\end{equation}
For two functions $f,g: \mathbb{N}\to\mathbb{R}$ we denote $f(N) \asymp g(N)$ to mean $c \le \vert{}f(N)/g(N)\vert{} \le C$ for constants $0 < c < C < \infty$  and for any $N$ sufficiently large.

For random variables $X_1,\dots,X_r$ with finite moments, the joint cumulant is
\[
\cum(X_1,\dots,X_r)=\sum_{\pi\in\mathcal P_r}(|\pi|-1)!\,(-1)^{|\pi|-1}\prod_{B\in\pi}\E\Big[\prod_{i\in B}X_i\Big],
\]
$\mathcal P_r$ being the partitions of $\{1,\dots,r\}$.  For a stationary field $\phi^N$
we write $\kappa_r^{(N)}(x_1,\dots,x_r):=\cum(\phi^N_{x_1},\dots,\phi^N_{x_r})$.  Cumulants
are multilinear, vanish (for $r\ge3$) on Gaussian vectors, and a centered law is Gaussian
if and only if all cumulants of order $\ge3$ vanish; see \cite{peccati_wiener_2011,shiryaev_probability-2_2019}.

For $f\in C^\infty(\T^d)$ and $k\in\Z^d$ we set
\[
\widehat f(k):=\int_{\T^d}f(x)\,e^{-2\pi i k\cdot x}\dd x,
\qquad
f(x)=\sum_{k\in\Z^d}\widehat f(k)\,e^{2\pi i k\cdot x},
\]
and write $e_k(x):=e^{2\pi i k\cdot x}$, so that $\{e_k\}_{k\in\Z^d}$ is an orthonormal
basis of $L^2(\T^d)$ and $\la f,e_k\ra_{L^2(\T^d)}=\widehat f(k)$.  Because $f$ is
smooth, $\widehat f$ decays faster than any polynomial:
\begin{equation}\label{eq:rapid-decay}
\forall M>0\quad \sup_{k\in\Z^d}(1+|k|)^M|\widehat f(k)|<\infty.
\end{equation}
Let $P$ be a scalar-valued constant-coefficient differential operator of exact order $m\ge0$,
\[
P=\sum_{|\alpha|=m}c_\alpha\,\partial^\alpha,
\qquad c_\alpha\in\R,
\qquad
\partial^\alpha:=\partial_1^{\alpha_1}\cdots\partial_d^{\alpha_d},
\quad\alpha\in\N^d,\quad |\alpha|=\textstyle\sum_j\alpha_j.
\]
Basic examples are
\begin{equation*}
\begin{split}
&P=\mathrm{Id}\ (m=0),
\qquad
P=\partial_v:=v\cdot\nabla\ (m=1,\ v\in\R^d),\\
&P=\Delta\ (m=2),
\qquad
P=(-\Delta)^{m/2}\ (m\ \text{even}),
\end{split}
\end{equation*}
and more generally any scalar homogeneous differential operator of order $m$.
The Fourier symbol is equal to
\begin{equation}\label{eq:symbol}
p(\xi):=\sum_{|\alpha|=m}c_\alpha\,(2\pi i\,\xi)^\alpha,
\qquad\xi\in\R^d,
\qquad\text{so that}\quad
\widehat{Pf}(k)=p(k)\,\widehat f(k),\ \ k\in\Z^d.
\end{equation}
Note that $p$ is homogeneous of degree $m$: $p(t\xi)=t^m p(\xi)$, and there is $C_p>0$
with $|p(\xi)|\le C_p|\xi|^m$.  We write $Z(p):=\{k\in\Z^d:p(k)=0 \}$; for $m\ge1$ one has
$0\in Z(p)$ since $p(0)=0$.

For $j=1,\dots,d$ let
$\nabla_{N,j}u_x:=u_{x+e_j}-u_x$ denote the forward difference on the discrete torus (where
periodicity is exact), set $\nabla_N^\alpha:=\prod_{j=1}^d\nabla_{N,j}^{\alpha_j}$, and
define the associated discrete operator
\begin{equation}\label{eq:def:PNdisc}
P_N:=\sum_{|\alpha|=m}c_\alpha\,\nabla_N^\alpha .
\end{equation}

For $s\in\R$ let
\[
H^{s}(\T^d):=\Big\{u\in\mathscr D'(\T^d):\ \|u\|_{H^s}^2:=\sum_{k\in\Z^d}(1+|k|^2)^{s}|\widehat u(k)|^2<\infty\Big\}
\]
($\mathscr D'(\T^d)$ the space of distributions on the torus, the topological dual of $C^\infty(\T^d)$).
The pairing of $u\in H^{-s}$ with $f\in C^\infty(\T^d)$ is
$\la u,f\ra:=\sum_k\widehat u(k)\,\overline{\widehat f(k)}$, conjugate-linear in $f$.

All fluctuation fields in this paper are real random distributions, and the defining formula \eqref{eq:def:PN} is primarily a statement about real test functions; scomplex test functions enter only through the Fourier modes $e_k$.  We fix the
conventions once and for all.  Since $P$ has real coefficients,
$\XN(\bar f)=\overline{\XN(f)}$, and we identify the random distribution with the
functional \eqref{eq:def:PN} through
\[
\la\XN,f\ra:=\XN(\bar f),
\]
which is conjugate-linear in $f$ and agrees with $\XN(f)$ for real $f$, consistently
with the pairing above.  For complex square-integrable random variables we write
\[
\Cov(A,B):=\E\big[\overline{(A-\E A)}\,(B-\E B)\big],
\]
conjugate-linear in the first argument, and $\Var(A):=\Cov(A,A)=\E|A-\E A|^2$.  With
these conventions every covariance formula below is Hermitian and reduces, for real
test functions, to the usual real covariance; statements about convergence in law
(Lemma~\ref{lem:fdd}) are proved for real test functions and extended to complex ones
by splitting into real and imaginary parts.

For a centered stationary field $(\phi_x)_{x\in\Z^d}$ let
$\Gamma_\phi(z):=\Cov(\phi_0,\phi_z)$ denote the covariance function.  If $\Gamma_\phi\in\ell^1(\Z^d)$, its spectral
density is well defined and given by
\begin{equation}
\widehat\Gamma_\phi(\theta):=\sum_{z\in \Z^d}\Gamma_\phi(z)e^{-2\pi i\theta\cdot z}, \qquad \theta\in\T^d.
\end{equation}  
The value
\begin{equation}\label{eq:susceptibility}
\chi_\phi:=\widehat\Gamma_\phi(0)=\sum_{z\in \Z^d}\Gamma_\phi(z)
\end{equation}
is the \emph{susceptibility} of the field; finiteness of $\chi_\phi$ is the classical
hypothesis under which the volume average satisfies a central limit theorem
\cite{newman_normal_1980}.  For an operator of order $m\ge1$, our main theorem does
not require finite susceptibility: the zero mode of the covariance is precisely compensated by 
the differentiation effect. We refer to Theorem~\ref{thm:main} and the discussion below
for more on this aspect.

\subsection{Two hypotheses and the main theorem}\label{sec:hypotheses}

We collect the standing assumptions on the field $(\phi^N_x)$.  On the finite torus $\TN$, the
spectral measure of $\Gamma_N$ is atomic and we assume it is the restriction of a fixed function $\widehat\Gamma_N$ on $\T^d$, i.e.
\begin{equation}\label{eq:disc-bochner}
\Gamma_N(z)=N^{-d}\sum_{k\in\widehat\TN}\widehat\Gamma_N(k/N)\,e^{2\pi i k\cdot z/N},\qquad z\in\TN ,
\end{equation}
the prefactor $N^{-d}$ making this the Riemann-sum discretization of $\Gamma_N(z)=\int_{\T^d}\widehat\Gamma_N(\theta)e^{2\pi i\theta\cdot z}\dd\theta$ over the
$N^d$ dual frequencies. Such hypothesis is natural in the torus setting, see, e.g., Example \ref{ex:gff-gradient}, where $\widehat\Gamma$ is explicit.

\begin{assumption}[Infrared scaling of the covariance]\label{ass:spectral}
There exist $\gamma:\Z^d\setminus\{0\}\to[0,\infty)$ and constants $C,M\ge0$
such that, for every $k\in\Z^d\setminus\{0\}$,
\begin{equation}\label{eq:spectral-scaling}
\lim_{N\to\infty}N^{-2m}\,\widehat\Gamma_N(k/N)=\gamma(k),
\end{equation}
and, for all $N\ge1$ and all $k\in\widehat\TN$ with $1\le|k|$,
\begin{equation}\label{eq:domination}
N^{-2m}\,\widehat\Gamma_N(k/N)\le C\,(1+|k|)^{M}.
\end{equation}
Here $m$ is the order of the operator $P$ fixed in \eqref{eq:def:PN}.
At the zero mode we assume: if $m=0$, the limit \eqref{eq:spectral-scaling} holds also at
$k=0$ with $\gamma(0)<\infty$ (finite susceptibility); if $m\ge1$,
\begin{equation}\label{eq:zero-mode}
\widehat\Gamma_N(0)\le C\,N^{2m+M},
\end{equation}
and we set $\gamma(0):=0$ by convention, so that sums over all $k\in\Z^d$ in the limit
formulas below are well defined.
\end{assumption}

The assumption constrains $\widehat\Gamma_N$ only at the sampled (integer) frequencies,
which is all the finite-torus fields determine; the whole-space variant
(Theorem~\ref{thm:wholespace}) carries the analogous hypotheses at real frequencies.
For $m\ge1$ no limiting value of $\widehat\Gamma_N$ at zero is required but the polynomial control \eqref{eq:zero-mode}.  In particular,  $\gamma(k)$ may blow up as $k\to0$. For $m=0$ it is instead the finite-susceptibility hypothesis $\gamma(0)<\infty$ of Assumption~\ref{ass:spectral} that governs.
The growth exponent $M$ in \eqref{eq:domination} need not be sharp; it only controls the Sobolev regularity of the limit.  When $\gamma$ is locally bounded (e.g.\ $P=\mathrm{Id}$ and a bounded spectral
density), one may take $M=0$.  In this regime, the profile $\gamma$ need not be constant: $\gamma(k)=\lim_N\widehat\Gamma_N(k/N)$ is constant precisely when $\widehat\Gamma_N$ converges continuously at the origin (the classical finite-susceptibility case, where $\gamma\equiv\chi$); a bounded spectral density with a directional limit
$\widehat\Gamma(\theta)\to\Psi(\theta/|\theta|)$ as $\theta\to0$ yields the direction-dependent profile $\gamma(k)=\Psi(k/|k|)$.

\begin{assumption}[Vanishing higher cumulants]\label{ass:cumulants}
For every $r\ge3$,
\begin{equation}\label{eq:higher-cumulants}
\lim_{N\to\infty}N^{-r(d/2+m)}\sum_{x_1,\dots,x_r\in\LN}\big|\kappa_r^{(N)}(x_1,\dots,x_r)\big|=0.
\end{equation}
\end{assumption}

We are ready to state our main theorem.

\begin{theorem}\label{thm:main}
Let $P$ be a scalar constant-coefficient differential operator of exact order $m\ge0$
with symbol $p$, and let $(\phi^N_x)$ be centered and stationary, with finite moments of
all orders, satisfying Assumptions~\ref{ass:spectral} and~\ref{ass:cumulants}.  Then, for every
$s>\tfrac d2+m+\tfrac M2$, the random distributions $\XN$ defined in \eqref{eq:def:PN}
converge in law in $H^{-s}(\T^d)$ to the centered Gaussian random distribution
$\mathcal X^P$ characterized by
\begin{equation}\label{eq:limit-cov}
\E\big[\la\mathcal X^P,f\ra\,\overline{\la\mathcal X^P,g\ra}\big]
=\sum_{k\in\Z^d}\overline{\widehat f(k)}\,\widehat g(k)\,|p(k)|^2\,\gamma(k),
\qquad f,g\in C^\infty(\T^d).
\end{equation}
\end{theorem}
The theory tackles two genuinely different regimes, according to the zero-frequency
behaviour of the covariance. When the spectral density is
singular at the origin at the rate that makes
$N^{-2m}\widehat\Gamma_N(\theta/N)\to\gamma(\theta)$ a nontrivial profile, the order-zero
volume average is not variance-normalized (its variance diverges and the susceptibility is
infinite), but the macroscopic order-$m$ observable $\XN$ of \eqref{eq:def:PN} is nondegenerate: the operator multiplies
the spectrum by $|p(k)|^2$, which for elliptic homogeneous $P$ is comparable to
$|k|^{2m}$ and suppresses exactly the divergent low-frequency mass (for directional or degenerate operators the suppression occurs only away from the set $\{p=0\}$).  This is the
regime of Theorem~\ref{thm:main} for $m\ge1$, where Assumption~\ref{ass:spectral} requires
no finite susceptibility and controls the zero mode only through the growth bound
\eqref{eq:zero-mode}; it covers the gradient free field (see Example~\ref{ex:gff-gradient} below) and the long-memory regime of Section~\ref{sec:paradigm}, and it is the mechanism behind
gradient observables of $\nabla\varphi$-interface models and behind the low-temperature
limits of Newman and Wu \cite{newman_gaussian_2018}.
When instead the spectral densities are uniformly finite near the zero mode, i.e.\ for
some $\delta>0$
\[
\sup_{N}\ \sup_{|\theta|\le\delta}\widehat\Gamma_N(\theta)<\infty,
\]
and $\widehat\Gamma_N(0)\to\chi$, so that the susceptibility \eqref{eq:susceptibility} is
finite in the limit, one has
$N^{-2m}\widehat\Gamma_N(\theta/N)\to0$ for $m\ge1$, so the only non-trivial limit is given by $m=0$ and $P$ being the identity operator.

We work on the torus so that Fourier series keep every step exact and elementary.  In
the applications the field lives on all of $\Z^d$ and the natural test functions are
$f\in C_c^\infty(\R^d)$; this whole-space setting is the more common one in the
literature.  We now state the whole-space form of the theory as a theorem in its own
right; its proof is given in \S\ref{sec:ws-proof}.

For each $N$ let $(\phi^N_x)_{x\in\Z^d}$ be centered and stationary with finite moments
of all orders (a fixed field $\phi^N=\phi$ is allowed, but so is a genuine triangular
array), whose covariance $\Gamma_N(z):=\Cov(\phi^N_0,\phi^N_z)$ has a spectral density
$S_N\in L^1(\T^d)$,
\[
\Gamma_N(z)=\int_{\T^d}S_N(\theta)\,e^{2\pi i\theta\cdot z}\dd\theta,\qquad z\in\Z^d.
\]
The relevant observable is
\begin{equation}\label{eq:def:PN-ws}
\XN(f):=N^{-d/2-m}\sum_{x\in\Z^d}\phi^N_x\,(Pf)(x/N),
\qquad f\in C_c^\infty(\R^d),
\end{equation}
the sum being finite: $(Pf)(x/N)\ne0$ forces $x\in N\supp(Pf)\cap\Z^d$.  Set
\begin{equation}\label{eq:ws-GN}
Q_N:=(-N/2,N/2]^d,\qquad
G_N(\xi):=\mathbf 1_{Q_N}(\xi)\,N^{-2m}S_N(\xi/N),\qquad\xi\in\R^d.
\end{equation}

\begin{assumption}[Whole-space hypotheses]\label{ass:wholespace}
There exist a measurable $\gamma:\R^d\setminus\{0\}\to[0,\infty)$ and constants
$C,M\ge0$ such that:
\begin{enumerate}[label=\textup{(W\arabic*)}]
\item\label{it:ws-scaling} $G_N(\xi)\to\gamma(\xi)$ for almost every $\xi\in\R^d$;
\item\label{it:ws-domination} $0\le G_N\le\Phi$ almost everywhere, for all $N$, where
$\Phi$ is locally integrable and $\Phi(\xi)\le C(1+|\xi|)^{M}$ for $|\xi|\ge1$;
\item\label{it:ws-cumulants} for every compact $K\subset\R^d$ and every $r\ge3$,
\begin{equation}\label{eq:ws-cum}
N^{-r(d/2+m)}\sum_{x_1,\dots,x_r\in NK\cap\Z^d}
\big|\kappa_r^{(N)}(x_1,\dots,x_r)\big|\longrightarrow0 .
\end{equation}
\end{enumerate}
\end{assumption}

Conditions \ref{it:ws-scaling}--\ref{it:ws-domination} are the real-frequency analogues
of Assumption~\ref{ass:spectral}; local integrability of $\Phi$ replaces the torus
zero-mode condition and controls the continuum of frequencies near the origin (it has no
torus counterpart, the nonzero dual frequencies being integers).  Note that
$\gamma\le\Phi$ almost everywhere, so
$\int|p(\xi)|^2|\widehat f(\xi)|^2\gamma(\xi)\dd\xi<\infty$ for every
$f\in C_c^\infty(\R^d)$, by the rapid decay of $\widehat f$ and the polynomial growth of
$\Phi$.  Condition \ref{it:ws-cumulants} is Assumption~\ref{ass:cumulants} localized to
growing windows $NK$.

\begin{theorem}[Whole-space form of Theorem \ref{thm:main}]\label{thm:wholespace}
Let $P$ be as in Theorem~\ref{thm:main} and let $(\phi^N_x)_{x\in\Z^d}$ be as above,
satisfying Assumption~\ref{ass:wholespace}.  Then, for all $f,g\in C_c^\infty(\R^d)$,
\begin{equation}\label{eq:ws-cov-limit}
\Cov\big(\XN(f),\XN(g)\big)\longrightarrow
\int_{\R^d}\overline{\widehat f(\xi)}\,\widehat g(\xi)\,|p(\xi)|^2\gamma(\xi)\dd\xi,
\end{equation}
and, for every $s>d/2+m+M/2$, the fields $\XN$ of \eqref{eq:def:PN-ws} converge in law
in $H^{-s}_{\mathrm{loc}}(\R^d)$ (i.e.\ $\psi\XN\to\psi\mathcal X^P$ in law in
$H^{-s}(\R^d)$ for every $\psi\in C_c^\infty(\R^d)$) to the centered Gaussian random
distribution $\mathcal X^P$ with covariance \eqref{eq:ws-cov-limit}.
\end{theorem}

\begin{remark}\label{rem:ws-usage}
For a fixed field with spectral density $S$, \ref{it:ws-scaling} reads
$N^{-2m}S(\xi/N)\to\gamma(\xi)$ a.e., and \ref{it:ws-domination} is a global majorant
for the rescaled density. Condition \eqref{eq:ws-cum} follows from the $\Lambda_N$-version
\eqref{eq:higher-cumulants} by stationarity, the sum over $NK\cap\Z^d$ being bounded by
the sum over a box of side $\lceil CN\rceil$ containing $NK$, and from the tree--graph criterion in the form of
Remark~\ref{rem:treegraph-Zd}.  It is in this form that the theory is applied in
\S\ref{sec:applications}.
\end{remark}

\begin{remark}\label{rem:vector}
If $P=(P^{(1)},\dots,P^{(n)})$ has scalar components $P^{(i)}$ of common order $m$ and symbols $p^{(i)}$, then $\XN$ is $\R^n$-valued and Theorem~\ref{thm:main} holds verbatim by the Cram\'er--Wold device \cite{cramer_theorems_1936} applied to arbitrary linear combinations $\sum_i a_i\,\la\XN,f\ra^{(i)}$, the limiting covariance being the matrix multiplier with entries
$\overline{p^{(i)}(k)}\,p^{(j)}(k)\,\gamma(k)$.  In particular $P=\nabla$ ($p^{(i)}(k)=2\pi i k_i$) and $P=\nabla^{\otimes m}$ (tensor-valued, with component symbols $p^{(\alpha)}(k)=(2\pi i k)^{\alpha}$ for multi-indices $|\alpha|=m$) are covered; when $\gamma(k)\sim c_*\,|k|^{-2m}$ the $m$-th order symbols cancel the infrared singularity and $|p^{(\alpha)}(k)|^2\gamma(k)$ stays bounded.  The limit of the $m$-th gradient observable is then \emph{projected} tensor white noise, with covariance $c_*(2\pi)^{2m}\sum_{k\ne0}\overline{(ik)^{\alpha}}(ik)^{\alpha'}|k|^{-2m}\,
\overline{\widehat f(k)}\,\widehat g(k)$: for $m=1$ the matrix multiplier is proportional to $k\otimes k/|k|^2$ ($k\ne0$), so the limit is the gradient of the continuum mean-zero Gaussian free field and \emph{not} an ordinary $d$-component white noise (it is isonormal white noise only after restricting the test space to gradient vector fields).
\end{remark}

\subsection{White noise by exact cancellation: the gradient of the free field}

Recall that $Z(p):=\{k\in\Z^d:p(k)=0\}$. We are ready to state the white-noise result as known, for instance, for the Breuer-Major theorem \cite{breuer_central_1983, coppini_central_2025}.

\begin{corollary}[White noise under exact cancellation]\label{cor:white-noise}
Under the hypotheses of Theorem~\ref{thm:main}, suppose there is $c_*>0$ with
\[
|p(k)|^2\,\gamma(k)=c_*\qquad\text{for all }k\in\Z^d\setminus Z(p).
\]
When $P=\mathrm{Id}$, constancy forces $\gamma(k)=c_*$ for every $k\in\Z^d$, and $c_*=\lim_N\widehat\Gamma_N(0)$
is the limiting susceptibility of the fields.  Then
\[
\E\big[\la\mathcal X^P,f\ra\,\overline{\la\mathcal X^P,g\ra}\big]
=c_*\sum_{k\in\Z^d\setminus Z(p)}\overline{\widehat f(k)}\,\widehat g(k).
\]
In particular, if $P=\mathrm{Id}$ then $\XN\xrightarrow{d}\sqrt{c_*}\,W$ in $H^{-s}(\T^d)$
for every $s>d/2+M/2$, with $W$ standard white noise on $\T^d$; and if $Z(p)=\{0\}$ (e.g.
$P=\partial_v$ in $d=1$, or any elliptic homogeneous $p$ vanishing only at $0$), the same
holds on the closed subspace of mean-zero test functions, in $H^{-s}(\T^d)$ for every
$s>d/2+m+M/2$.
\end{corollary}

\begin{proof}[Proof of Corollary~\ref{cor:white-noise}]
Insert $|p(k)|^2\gamma(k)=c_*$ for
$k\notin Z(p)$ into \eqref{eq:limit-cov}; the terms with $k\in Z(p)$ carry the factor
$|p(k)|^2=0$ and drop out.  When $P=\mathrm{Id}$, $Z(p)=\emptyset$ and the covariance is
$c_*\sum_k\overline{\widehat f(k)}\widehat g(k)=c_*\la f,g\ra_{L^2(\T^d)}$,
the covariance of $\sqrt{c_*}\,W$.  When $Z(p)=\{0\}$ the same holds on mean-zero test
functions.

The topology is the one supplied by Theorem~\ref{thm:main}: exact cancellation
constrains only the limiting fixed frequencies and gives no ultraviolet control on the
finite-$N$ modes, so it cannot improve the Sobolev index of the tightness estimate.  For
instance, Gaussian fields with all fixed low-frequency spectral weights equal to $1$ and
one symmetric pair of modes $\pm k_N$, $|k_N|\asymp N$, of weight $N^{M}$ satisfy
Assumptions~\ref{ass:spectral}--\ref{ass:cumulants} with $\gamma\equiv1$ and exact
cancellation for $P=\mathrm{Id}$, yet the $H^{-s}$ mass of that pair is of order
$N^{M-2s}$, so the sequence is not tight in $H^{-s}$ when $M>2s$.  When the spectral
densities are uniformly bounded one may take $M=0$, and for $P=\mathrm{Id}$ the
convergence then holds for every $s>d/2$.
\end{proof}

We illustrate the exact-cancellation regime of Corollary~\ref{cor:white-noise} on the Gaussian
free field, whose infinite susceptibility is compensated by a single derivative;
finite-susceptibility examples are the Gibbs fields of \S\ref{sec:dobrushin-tree-graph}.

\begin{example}[Gradient of the GFF: infinite susceptibility]\label{ex:gff-gradient}
Let $\phi^N$ be the lattice GFF on $\TN$ at inverse temperature
$\beta\in(0,\infty)$: the centered Gaussian field with
$\widehat\Gamma_N(\theta)=(\beta\lambda_N(\theta))^{-1}$ for $\theta\ne0$ and
$\widehat\Gamma_N(0):=0$, where $\lambda_N(\theta):=\sum_{j=1}^d4\sin^2(\pi\theta_j)$ is
the symbol of the discrete Laplacian (zero mode removed: the field is defined modulo
constants).  The spectral density blows up near the zero mode,
$\widehat\Gamma_N(k/N)\sim N^2/(\beta(2\pi)^2|k|^2)$ for fixed $k\ne0$, so the $m=0$ case of
Assumption~\ref{ass:spectral} fails---there is no finite limiting susceptibility---and it can be shown (by Equation \eqref{eq:cov-exact} below) that the variance of the order-zero average $\mathcal X_N^{\mathrm{Id}}(f)$ diverges like $N^{2}$ for every non-constant $f$.  Assumption~\ref{ass:spectral} instead holds with $m=1$: for the scaling \eqref{eq:spectral-scaling},
\[
N^{-2}\widehat\Gamma_N(k/N)=\frac{N^{-2}}{\beta\lambda_N(k/N)}
\longrightarrow\gamma(k)=\frac{1}{\beta\,(2\pi)^2|k|^2}\qquad(k\ne0),
\]
since $\lambda_N(k/N)=\sum_j4\sin^2(\pi k_j/N)\sim(2\pi)^2|k|^2/N^2$, confirming that
$N^{2}$ (i.e.\ the normalization $N^{-d/2-1}$ in \eqref{eq:def:PN}) is the correct
scaling; the domination \eqref{eq:domination} holds with $M=0$, since
$\sin^2(\pi t)\ge4t^2$ for $|t|\le\tfrac12$ gives
$N^{-2}\widehat\Gamma_N(k/N)\le(16\,\beta|k|^2)^{-1}\le C$ for all $k\in\widehat\TN$
with $1\le|k|$; and the zero-mode bound \eqref{eq:zero-mode} is trivial,
$\widehat\Gamma_N(0)=0$.  The field is Gaussian, so Assumption~\ref{ass:cumulants} is
immediate.  The profile $\gamma(k)=c_*|k|^{-2}$ with $c_*=1/(\beta(2\pi)^2)$ is exactly the
cancellation regime of Remark~\ref{rem:vector}: for $P=\nabla$, with symbols
$p^{(i)}(k)=2\pi ik_i$,
\[
\sum_{i=1}^d|p^{(i)}(k)|^2\,\gamma(k)
=(2\pi)^2|k|^2\cdot\frac{1}{\beta\,(2\pi)^2|k|^2}=\frac1\beta\qquad(k\ne0),
\]
and the gradient observable converges on the mean-zero subspace to the projected white noise of Remark~\ref{rem:vector},
\[
\E\big[\la\mathcal X^{\nabla,i},f\ra\,\overline{\la\mathcal X^{\nabla,j},g\ra}\big]
=\frac1\beta\sum_{k\ne0}\frac{k_ik_j}{|k|^2}\,\overline{\widehat f(k)}\,\widehat g(k).
\]
The limit is $\beta^{-1/2}$ times the gradient of the
continuum mean-zero Gaussian free field, acting as white noise of intensity $1/\beta$ on
gradient test fields.
This is the exactly Gaussian, torus analogue of the low-temperature gradient white-noise limits of Newman and Wu \cite{newman_gaussian_2018}. 
\end{example}

\section{Proofs of the main results}\label{sec:proofs}

The proof of Theorem~\ref{thm:main} has three steps: (i) an exact computation of the
covariance on the discrete torus (Lemma~\ref{lem:alias}) and its limit
(Proposition~\ref{prop:cov-limit}); (ii) the asymptotic vanishing of cumulants of order
$\ge3$; (iii) tightness in $H^{-s}$.

\subsection{Exact covariance on the discrete torus}

Throughout this section we work under the modelling convention of \eqref{eq:disc-bochner}:
the spectral weights of the discrete torus are the samples $\widehat\Gamma_N(k/N)$,
$k\in\widehat\TN$ (recall that $\TN=(-N/2,N/2]^d\cap\Z^d$) of a fixed density on $\T^d$.  This convention holds exactly for a field arising as the periodization of a stationary field on $\Z^d$ with absolutely
summable covariance $\Gamma$ and spectral density $\widehat\Gamma$: if
$\Gamma_N(z):=\sum_{n\in\Z^d}\Gamma(z+Nn)$, then its discrete Fourier weight is
\[
\sum_{z\in\LN}\Gamma_N(z)\,e^{-2\pi ik\cdot z/N}
=\sum_{u\in\Z^d}\Gamma(u)\,e^{-2\pi ik\cdot u/N}
=\widehat\Gamma(k/N),
\]
so \eqref{eq:disc-bochner} holds with $\widehat\Gamma_N=\widehat\Gamma$, with no error
term.

\begin{lemma}[Discrete Fourier representation]\label{lem:alias}
For $f,g\in C^\infty(\T^d)$,
\begin{equation}\label{eq:cov-exact}
\Cov\big(\XN(f),\XN(g)\big)
=\sum_{k\in\widehat\TN}\overline{\widehat{Pf}^{(N)}(k)}\,\widehat{Pg}^{(N)}(k)\;
N^{-2m}\,\widehat\Gamma_N\!\big(k/N\big),
\end{equation}
a \emph{finite} sum over the dual group $\widehat\TN=(-N/2,N/2]^d\cap\Z^d$, where
\[
\widehat{Ph}^{(N)}(k):=N^{-d}\sum_{x\in\LN}(Ph)(x/N)\,e^{-2\pi i k\cdot x/N}
\]
is the discrete Fourier coefficient.
\end{lemma}

\begin{proof}
By definition \eqref{eq:def:PN} and stationarity,
\[
\Cov\big(\XN(f),\XN(g)\big)
=N^{-d-2m}\sum_{x,y\in\LN}\Gamma_N(x-y)\,\overline{(Pf)(x/N)}\,(Pg)(y/N).
\]
Insert the spectral representation \eqref{eq:disc-bochner} for $\Gamma_N(x-y)$ and the
discrete Fourier inversions
$(Ph)(x/N)=\sum_{k'\in\widehat\TN}\widehat{Ph}^{(N)}(k')\,e^{2\pi i k'\cdot x/N}$ for
$h=f,g$.  The orthogonality relation
\begin{equation}\label{eq:disc-orth}
N^{-d}\sum_{x\in\LN}e^{2\pi i(k-k')\cdot x/N}=\mathbf 1_{\{k\equiv k'\ (\mathrm{mod}\ N)\}},
\qquad k,k'\in\widehat\TN ,
\end{equation}
collapses the sums over $x$ and over $y$. Since we work in$\widehat\TN$, only the diagonal terms survive, and collecting them gives
\eqref{eq:cov-exact}.
\end{proof}

\begin{proposition}[Covariance limit]\label{prop:cov-limit}
Under Assumption~\ref{ass:spectral},
\begin{equation}\label{eq:cov-converges}
\lim_{N\to\infty}\Cov\big(\XN(f),\XN(g)\big)
=\sum_{k\in\Z^d}\overline{\widehat f(k)}\,\widehat g(k)\,|p(k)|^2\,\gamma(k).
\end{equation}
\end{proposition}

\begin{proof}
\emph{Step 1: aliasing.}  We first show that for every $h\in C^\infty(\T^d)$ the discrete
coefficients of Lemma~\ref{lem:alias} are the exact aliased sums
\begin{equation}\label{eq:alias}
\widehat{Ph}^{(N)}(k)=\sum_{\ell\in\Z^d}\widehat{Ph}(k+N\ell),\qquad
\widehat{Ph}(\zeta)=p(\zeta)\,\widehat h(\zeta),
\end{equation}
i.e., an equivalent statement to Poisson summation on $\LN$.  Since $Ph\in C^\infty(\T^d)$, the rapid decay
\eqref{eq:rapid-decay} applied to $h$ gives, for every $M''>0$,
\begin{equation}\label{eq:Ph-decay}
|\widehat{Ph}(\zeta)|=|p(\zeta)|\,|\widehat h(\zeta)|
\le C_p|\zeta|^m\,|\widehat h(\zeta)|
\le C_{h,M''}\,(1+|\zeta|)^{-M''},
\end{equation}
so the Fourier expansion $(Ph)(\theta)=\sum_{n\in\Z^d}\widehat{Ph}(n)\,e^{2\pi in\cdot\theta}$
converges absolutely.  Substituting it into the definition of $\widehat{Ph}^{(N)}(k)$ and
interchanging the two (absolutely convergent) sums,
\[
\widehat{Ph}^{(N)}(k)
=\sum_{n\in\Z^d}\widehat{Ph}(n)\,N^{-d}\sum_{x\in\LN}e^{2\pi i(n-k)\cdot x/N}
=\sum_{n\in\Z^d}\widehat{Ph}(n)\,\mathbf 1_{\{n\equiv k\ (\mathrm{mod}\ N)\}},
\]
by the orthogonality relation \eqref{eq:disc-orth}, which holds for arbitrary $n\in\Z^d$
because both sides are $N$-periodic in $n$; writing the surviving indices as $n=k+N\ell$,
$\ell\in\Z^d$, gives \eqref{eq:alias}.

\emph{Step 2: uniform decay.}  We next claim that on the dual box $\widehat\TN$,
\begin{equation}\label{eq:disc-decay}
|\widehat{Ph}^{(N)}(k)|\le C_{h,M'}(1+|k|)^{-M'}\quad\text{uniformly in }N,\ \text{for every }M',
\end{equation}
with a constant depending also on $d$.
Fix $M'>0$ and use \eqref{eq:Ph-decay} with $M''=M'+d+1$.  The $\ell=0$ term of
\eqref{eq:alias} is bounded by $C_{h,M''}(1+|k|)^{-M'-d-1}\le C_{h,M''}(1+|k|)^{-M'}$.  For
the aliasing tail, every $k\in\widehat\TN$ satisfies $|k_j|\le N/2$ coordinatewise, so for
$\ell\ne0$,
\[
|k_j+N\ell_j|\ge N|\ell_j|-N/2\ge\frac N2|\ell_j|\quad(\ell_j\ne0),
\]
hence
\[
|k+N\ell|^2\ge\sum_{j:\ell_j\ne0}\Big(\frac N2\Big)^2\ell_j^2=\Big(\frac N2\Big)^2|\ell|^2,
\]
so that
\[
\sum_{\ell\ne0}\big|\widehat{Ph}(k+N\ell)\big|
\le C_{h,M''}\sum_{\ell\ne0}\Big(\frac{N|\ell|}{2}\Big)^{-M'-d-1}
=C_{h,M''}\Big(\frac2N\Big)^{M'+d+1}\sum_{\ell\ne0}|\ell|^{-M'-d-1}
\le C'\,N^{-M'},
\]
the last series converging because $M'+d+1>d$.  Since $1+|k|\le1+\sqrt d\,N/2\le\sqrt d\,N$
on $\widehat\TN$ (for $N\ge2$), the tail is also bounded by $C'_d(1+|k|)^{-M'}$, and
\eqref{eq:disc-decay} follows.  Moreover, for fixed $k$ the tail is $O(N^{-M'})$, so
\begin{equation}\label{eq:disc-conv}
\widehat{Ph}^{(N)}(k)\longrightarrow\widehat{Ph}(k)=p(k)\,\widehat h(k)
\qquad\text{as }N\to\infty,\ k\ \text{fixed}.
\end{equation}

\emph{Step 3: passage to the limit.}  The sum \eqref{eq:cov-exact} runs over the dual box
$\widehat\TN$; write it as $\sum_{k\in\Z^d}\mathbf 1_{\widehat\TN}(k)\,T_N(k)$ with
$T_N(k):=\overline{\widehat{Pf}^{(N)}(k)}\,\widehat{Pg}^{(N)}(k)\,N^{-2m}\widehat\Gamma_N(k/N)$.
Fix $k\in\Z^d\setminus\{0\}$.  For $N>2|k|$, $k$ lies in $\widehat\TN$ and, by \eqref{eq:disc-conv},
$\widehat{Pf}^{(N)}(k)\to p(k)\widehat f(k)$ and $\widehat{Pg}^{(N)}(k)\to p(k)\widehat g(k)$,
while \eqref{eq:spectral-scaling} gives $N^{-2m}\widehat\Gamma_N(k/N)\to\gamma(k)$; thus
$\mathbf 1_{\widehat\TN}(k)T_N(k)\to|p(k)|^2\gamma(k)\,\overline{\widehat f(k)}\,\widehat g(k)$
for every $k\ne0$.

The zero mode requires its own estimate when $m\ge1$: the discrete coefficient at $k=0$
is \emph{not} $p(0)\widehat f(0)$ but the alias sum \eqref{eq:alias}, in which only the
$\ell=0$ term vanishes.  By \eqref{eq:alias} and $p(0)=0$,
$\widehat{Ph}^{(N)}(0)=\sum_{\ell\ne0}\widehat{Ph}(N\ell)$, so \eqref{eq:Ph-decay} with
an exponent $M''>\max(d,M/2)$ gives
$|\widehat{Ph}^{(N)}(0)|\le C_{h,M''}\sum_{\ell\ne0}(N|\ell|)^{-M''}\le C\,N^{-M''}$
for $h=f,g$; together with the zero-mode bound \eqref{eq:zero-mode},
\[
|T_N(0)|\le C\,N^{-2M''}\cdot N^{-2m}\widehat\Gamma_N(0)\le C\,N^{M-2M''}\longrightarrow0 ,
\]
so $T_N(0)\to0=|p(0)|^2\gamma(0)$, consistently with the convention $\gamma(0)=0$ of
Assumption~\ref{ass:spectral}.  For $m=0$,
$T_N(0)\to|p(0)|^2\gamma(0)\,\overline{\widehat f(0)}\,\widehat g(0)$ directly, by
\eqref{eq:disc-conv} and the zero-mode part of Assumption~\ref{ass:spectral}.

For the interchange of limit and summation over $k\ne0$ we dominate uniformly in $N$: on
$\widehat\TN$, \eqref{eq:disc-decay} gives
$|\widehat{Pf}^{(N)}(k)|,\,|\widehat{Pg}^{(N)}(k)|\le C(1+|k|)^{-(M+d+1)/2-m}$ (choosing $M'$
large), and \eqref{eq:domination} gives $N^{-2m}\widehat\Gamma_N(k/N)\le C(1+|k|)^M$ for
$k\in\widehat\TN$ with $|k|\ge1$.  Hence $|\mathbf 1_{\widehat\TN}(k)T_N(k)|\le C'(1+|k|)^{-(d+1)}$ for $k\ne0$, summable over $\Z^d$, and dominated convergence for series yields \eqref{eq:cov-converges}.
\end{proof}

\subsection{Cumulants and finite-dimensional distributions}

\begin{lemma}\label{lem:fdd}
Under Assumption~\ref{ass:cumulants}, for every $f_1,\dots,f_r\in C^\infty(\T^d)$ with
$r\ge3$, we have that
\begin{equation*}
    \cum(\XN(f_1),\dots,\XN(f_r))\to0.
\end{equation*}
Consequently the finite-dimensional
distributions of $\XN$ converge to those of the centered Gaussian field with covariance
\eqref{eq:limit-cov}.
\end{lemma}

\begin{proof}
By multilinearity,
\[
\cum(\XN(f_1),\dots,\XN(f_r))
=N^{-r(d/2+m)}\sum_{x_1,\dots,x_r\in\LN}\kappa_r^{(N)}(x_1,\dots,x_r)\prod_{j=1}^r(Pf_j)(x_j/N).
\]
Bounding each $Pf_j$ by $\|Pf_j\|_\infty$ and using \eqref{eq:higher-cumulants} gives the
limit $0$.

It remains to upgrade the convergence of cumulants to convergence of laws; convergence
of all cumulants to the Gaussian ones does not, by itself, prove convergence of a
sequence of laws, and the missing step is the method of moments.  Let
$f_1,\dots,f_\ell\in C^\infty(\T^d)$ be real-valued, $a_1,\dots,a_\ell\in\R$, and set
$Z_N:=\sum_{j=1}^\ell a_j\XN(f_j)$.  By multilinearity, the above, Proposition~\ref{prop:cov-limit}
($r=2$) and centering ($r=1$), every cumulant of $Z_N$ converges to the corresponding
cumulant of the centered Gaussian law with variance
$\sigma^2:=\sum_{i,j}a_ia_j\sum_k\overline{\widehat f_i(k)}\widehat f_j(k)|p(k)|^2\gamma(k)$.
It is well known that moments can be expressed as cumulants polynomials, i.e., by the moment--cumulant formula,
\[
\E[Z_N^r]=\sum_{\pi\in\mathcal P_r}\prod_{B\in\pi}\cum_{|B|}(Z_N),
\]
so every moment of $Z_N$ converges to the corresponding moment of the Gaussian law
\cite[Cor.~3.2.2]{peccati_wiener_2011}: this ensures that $Z_N\Rightarrow\mathcal N(0,\sigma^2)$. Finally, the Cram\'er--Wold device \cite{cramer_theorems_1936} converts the convergence of all real linear combinations into convergence of the finite-dimensional distributions; complex test functions are recovered by splitting into real and imaginary parts.
\end{proof}

\subsection{Tightness}

\begin{proof}[Proof of Theorem~\ref{thm:main}]
By Lemma~\ref{lem:fdd} the finite-dimensional distributions converge to the claimed
Gaussian law.  It remains to prove tightness in $H^{-s}(\T^d)$ for $s>d/2+m+M/2$, after
which Prokhorov's theorem (tightness gives a relatively compact family on the separable
Hilbert space $H^{-s}(\T^d)$) together with the convergence of the finite-dimensional
distributions identifies the unique limit \cite[\S5]{billingsley_convergence_1999}.

Using $Pe_k=p(k)e_k$ and the exact covariance \eqref{eq:cov-exact} with $f=g=e_k$, by the
orthogonality relation \eqref{eq:disc-orth} the only principal frequency at which the
discrete coefficient of $Pe_k$ is nonzero is $\la k\ra_N$, where it equals $p(k)$;
hence
$\Var(\la\XN,e_k\ra)=|p(k)|^2\,N^{-2m}\widehat\Gamma_N(\la k\ra_N/N)$.  Using
$|p(k)|\le C_p(1+|k|)^m$, $|\la k\ra_N|\le|k|$, the domination \eqref{eq:domination},
and, for the residual modes $k\equiv0\ (\mathrm{mod}\,N)$, $k\ne0$ (which satisfy
$|k|\ge N$), the zero-mode part of Assumption~\ref{ass:spectral}, we obtain
\begin{equation}\label{eq:var-bound}
\Var\big(\la\XN,e_k\ra\big)\le C_*\,(1+|k|)^{2m+M},
\end{equation}
for a constant $C_*$ independent of $N,k$. We have used that for $m\ge1$, \eqref{eq:zero-mode} gives $N^{-2m}\widehat\Gamma_N(0)\le CN^{M}\le C(1+|k|)^{M}$, and for $m=0$ the convergence $\widehat\Gamma_N(0)\to\gamma(0)<\infty$ gives a uniform bound (when the zero mode is removed, $\widehat\Gamma_N(0)=0$, these terms are absent). 
Therefore, by Parseval in $H^{-s'}$, for every
$s'\in\R$,
\[
\sup_N\E\big[\|\XN\|_{H^{-s'}}^2\big]
=\sup_N\sum_{k\in\Z^d}(1+|k|^2)^{-s'}\,\Var\big(\la\XN,e_k\ra\big)
\le C_*\sum_{k\in\Z^d}(1+|k|^2)^{-s'}(1+|k|)^{2m+M},
\]
which is finite as soon as $2s'-(2m+M)>d$, i.e. $s'>d/2+m+M/2$, by comparison of the series
with the integral $\int_{\R^d}(1+|\xi|)^{2m+M-2s'}\dd\xi$ (the summand and the integrand
differ by a bounded factor on each unit cube $k+[0,1)^d$), which converges precisely when
$2s'-(2m+M)>d$.
Given $s>d/2+m+M/2$, pick $s'\in(d/2+m+M/2,\,s)$: the uniform bound holds in $H^{-s'}$,
and the embedding $H^{-s'}\hookrightarrow H^{-s}$ is compact. 
Tightness follows: by Markov's inequality the closed $H^{-s'}$-balls
$K_R:=\{\|u\|_{H^{-s'}}\le R\}$, which are compact in $H^{-s}$, satisfy
$\sup_N\PP(\XN\notin K_R)\le R^{-2}\sup_N\E\|\XN\|_{H^{-s'}}^2\to0$ as $R\to\infty$
\cite[\S5]{billingsley_convergence_1999}.  This completes the proof.
\end{proof}

\subsection{Proof of the whole-space theorem}\label{sec:ws-proof}

Throughout this subsection $f,g,h\in C_c^\infty(\R^d)$,
$\widehat h(\xi):=\int_{\R^d}h(t)e^{-2\pi i\xi\cdot t}\dd t$, and
\begin{equation}\label{eq:ws-AN}
A_{N,h}(\xi):=N^{-d}\sum_{x\in\Z^d}h(x/N)\,e^{-2\pi i\xi\cdot x/N},\qquad\xi\in\R^d,
\end{equation}
a finite sum.  Poisson summation applied to $x\mapsto h(x/N)e^{-2\pi i\xi\cdot x/N}$
gives the absolutely convergent alias representation
\begin{equation}\label{eq:ws-alias}
A_{N,h}(\xi)=\sum_{\ell\in\Z^d}\widehat h(\xi+N\ell).
\end{equation}

\begin{lemma}[Exact covariance and alias decay]\label{lem:ws-cov}
For all $f,g\in C_c^\infty(\R^d)$,
\begin{equation}\label{eq:ws-exact-cov}
\Cov\big(\XN(f),\XN(g)\big)
=\int_{Q_N}G_N(\xi)\,\overline{A_{N,Pf}(\xi)}\,A_{N,Pg}(\xi)\dd\xi .
\end{equation}
Moreover, for every $h\in C_c^\infty(\R^d)$:
\begin{enumerate}[label=\textup{(\alph*)}]
\item\label{it:ws-unif} $\sup_{N\ge1}\sup_{\xi\in\R^d}|A_{N,h}(\xi)|<\infty$;
\item\label{it:ws-decay} for every $M'>0$ there is $C=C(h,M',d)$ with
$|A_{N,h}(\xi)|\le C(1+|\xi|)^{-M'}$ for all $N\ge1$ and $\xi\in Q_N$;
\item\label{it:ws-conv} $\sup_{\xi\in Q_N}|A_{N,h}(\xi)-\widehat h(\xi)|=O(N^{-M'})$
for every $M'>0$.
\end{enumerate}
\end{lemma}

\begin{proof}
By \eqref{eq:def:PN-ws}, stationarity, and the convention
$\Cov(A,B)=\E[\overline{(A-\E A)}(B-\E B)]$,
\[
\Cov\big(\XN(f),\XN(g)\big)
=N^{-d-2m}\sum_{x,y\in\Z^d}\Gamma_N(x-y)\,\overline{(Pf)(x/N)}\,(Pg)(y/N),
\]
a finite double sum.  Inserting the spectral representation of $\Gamma_N$, interchanging
the (finite) sums with the integral, and substituting $\theta=\xi/N$
($\dd\theta=N^{-d}\dd\xi$, $\theta\in\T^d\leftrightarrow\xi\in Q_N$) gives
\begin{equation*}
\begin{split}
    &\Cov\big(\XN(f),\XN(g)\big)=\\
    &=\int_{Q_N}N^{-2m}S_N(\xi/N) \Big(\overline{N^{-d}\!\sum_x(Pf)(x/N)e^{-2\pi i\xi x/N}}\Big) \Big(N^{-d}\!\sum_y(Pg)(y/N)e^{-2\pi i\xi y/N}\Big)\dd\xi,
\end{split}
\end{equation*}
which is \eqref{eq:ws-exact-cov}.  

Turning to the other statements: \ref{it:ws-unif} holds since $|A_{N,h}(\xi)|\le N^{-d}\sum_x|h(x/N)|$, a Riemann sum of $\int|h|$.  For \ref{it:ws-decay}, fix $M'$ and use the fact that $h$ is notably Schwartz function and thus $|\widehat h(\zeta)|\le C_{h,M''}(1+|\zeta|)^{-M''}$ with $M''=M'+d+1$. Putting this in \eqref{eq:ws-alias}: the $\ell=0$ term is $\le C(1+|\xi|)^{-M'}$, and for $\ell\ne0$, coordinatewise $|\xi_j+N\ell_j|\ge N|\ell_j|-N/2\ge(N/2)|\ell_j|$ on $Q_N$, hence $|\xi+N\ell|\ge(N/2)|\ell|$, so
\[
\sum_{\ell\ne0}|\widehat h(\xi+N\ell)|
\le C_{h,M''}\Big(\frac2N\Big)^{M''}\sum_{\ell\ne0}|\ell|^{-M''}\le C'N^{-M'-d-1},
\]
and $1+|\xi|\le C_dN$ on $Q_N$ makes the tail $\le C''(1+|\xi|)^{-M'}$ as well.  The
same tail estimate gives \ref{it:ws-conv}.
\end{proof}

\begin{proof}[Proof of Theorem~\ref{thm:wholespace}]
\emph{Step 1: covariance.}  By Lemma~\ref{lem:ws-cov}\,\ref{it:ws-conv} and
\ref{it:ws-scaling}, the integrand of \eqref{eq:ws-exact-cov} converges almost
everywhere to
$\overline{\widehat{Pf}(\xi)}\,\widehat{Pg}(\xi)\gamma(\xi)
=\overline{\widehat f(\xi)}\,\widehat g(\xi)\,|p(\xi)|^2\gamma(\xi)$.  For a dominant,
split $\R^d$: on $|\xi|\le1$, Lemma~\ref{lem:ws-cov}\,\ref{it:ws-unif} bounds
$|\overline{A_{N,Pf}}A_{N,Pg}|$ by a constant and $G_N\le\Phi\in L^1(\{|\xi|\le1\})$; on
$|\xi|\ge1$, Lemma~\ref{lem:ws-cov}\,\ref{it:ws-decay} with $M'=(M+d+1)$ and
$G_N(\xi)\le C(1+|\xi|)^M$ give the integrable dominant $C(1+|\xi|)^{-d-1}$.  Dominated
convergence yields \eqref{eq:ws-cov-limit}, the limit being finite since
$\gamma\le\Phi$.

\emph{Step 2: cumulants and finite-dimensional distributions.}  Let
$f_1,\dots,f_r\in C_c^\infty(\R^d)$ be real, $r\ge3$, and let $K$ be a compact set
containing $\bigcup_j\supp f_j$, hence $\bigcup_j\supp(Pf_j)$.  By multilinearity,
\[
\big|\cum(\XN(f_1),\dots,\XN(f_r))\big|
\le\prod_{j=1}^r\|Pf_j\|_\infty\;
N^{-r(d/2+m)}\!\!\sum_{x_1,\dots,x_r\in NK\cap\Z^d}\!\!\big|\kappa_r^{(N)}(x_1,\dots,x_r)\big|
\longrightarrow0
\]
by \eqref{eq:ws-cum}.  Together with Step~1 ($r=2$) and centering ($r=1$), all joint cumulants converge to those of the Gaussian field with covariance \eqref{eq:ws-cov-limit}; the moment-method argument in the proof of Lemma~\ref{lem:fdd} applies verbatim and gives convergence of the finite-dimensional distributions of $\XN$ to those of $\mathcal X^P$.

\emph{Step 3: tightness in $H^{-s}_{\mathrm{loc}}$.}  Fix $\psi\in C_c^\infty(\R^d)$
and $s'\in(d/2+m+M/2,\,s)$.  The random distribution $\psi\XN$ is supported in the
fixed compact $\supp\psi$ and its Fourier transform is the function
$\widehat{\psi\XN}(\xi)=\XN(\psi e_{-\xi})$, $e_{-\xi}(t):=e^{-2\pi i\xi\cdot t}$.
Since $h_\xi:=\psi e_{-\xi}\in C_c^\infty(\R^d)$ with
$\widehat{Ph_\xi}(\zeta)=p(\zeta)\,\widehat\psi(\zeta+\xi)$, formula
\eqref{eq:ws-exact-cov} (with $f=g=h_\xi$) and \eqref{eq:ws-alias} give
\[
\E\big|\widehat{\psi\XN}(\xi)\big|^2
=\int_{Q_N}G_N(\eta)\,\big|A_N(\eta,\xi)\big|^2\dd\eta,
\qquad
A_N(\eta,\xi):=\sum_{\ell\in\Z^d}p(\eta+N\ell)\,\widehat\psi(\eta+N\ell+\xi).
\]
By Cauchy--Schwarz on the $\ell$-sum,
\[
|A_N(\eta,\xi)|^2
\le\Big(\sum_{\ell}|p(\eta+N\ell)|^2\,|\widehat\psi(\eta+N\ell+\xi)|\Big)
\Big(\sum_{\ell}|\widehat\psi(\eta+N\ell+\xi)|\Big),
\]
and the second factor is bounded by a constant $C_\psi$ uniformly in $N\ge1$, $\eta$,
$\xi$: the points $\{\eta+N\ell+\xi\}_\ell$ are $N$-separated and $\widehat\psi$ is
Schwartz, so $\sum_\ell(1+|\eta+N\ell+\xi|)^{-(d+1)}\le C_d$.  Hence, by Tonelli,
\[
\E\big[\|\psi\XN\|_{H^{-s'}}^2\big]
=\int_{\R^d}(1+|\xi|^2)^{-s'}\,\E\big|\widehat{\psi\XN}(\xi)\big|^2\dd\xi
\le C_\psi\int_{Q_N}\!G_N(\eta)\sum_\ell|p(\eta+N\ell)|^2\,I_\ell(\eta)\dd\eta,
\]
with $I_\ell(\eta):=\int(1+|\xi|^2)^{-s'}|\widehat\psi(\eta+N\ell+\xi)|\dd\xi$.  By
Peetre's inequality
$(1+|\xi|^2)^{-s'}\le2^{s'}(1+|\eta+N\ell|^2)^{-s'}(1+|\xi+\eta+N\ell|^2)^{s'}$, which implies that
\[
I_\ell(\eta)\le C(1+|\eta+N\ell|^2)^{-s'}\int(1+|\zeta|^2)^{s'}|\widehat\psi(\zeta)|\dd\zeta
=C'_\psi\,(1+|\eta+N\ell|^2)^{-s'} .
\]
With $|p(\zeta)|^2\le C(1+|\zeta|)^{2m}$ this yields, for $\eta\in Q_N$,
\begin{equation*}
\begin{split}
&\sum_\ell|p(\eta+N\ell)|^2 I_\ell(\eta)
\le C\sum_\ell(1+|\eta+N\ell|)^{2m-2s'}\\
&\le C\big[(1+|\eta|)^{2m-2s'}+N^{2m-2s'}\textstyle\sum_{\ell\ne0}|\ell|^{2m-2s'}\big]
\le C''(1+|\eta|)^{2m-2s'},
\end{split}
\end{equation*}
using $|\eta+N\ell|\ge(N/2)|\ell|$, the convergence of
$\sum_{\ell\ne0}|\ell|^{2m-2s'}$ (as $2s'-2m>d$), and $1+|\eta|\le C_dN$ with
$2m-2s'<0$.  Therefore
\[
\sup_N\E\big[\|\psi\XN\|_{H^{-s'}}^2\big]
\le C\int_{|\eta|\le1}\Phi(\eta)\dd\eta
+C\int_{|\eta|\ge1}(1+|\eta|)^{M+2m-2s'}\dd\eta<\infty,
\]
since $\Phi$ is locally integrable and $2s'>d+2m+M$.

Next, the embedding of $\{u\in H^{-s'}(\R^d):\supp u\subset K_0\}$ into $H^{-s}(\R^d)$
is compact for every compact $K_0$ and $s>s'$: if $(u_n)$ is bounded in $H^{-s'}$ with
supports in $K_0$, pick $\varphi\in C_c^\infty$ equal to $1$ on $K_0$; then
$\widehat u_n(\xi)=u_n(\varphi e_{-\xi})$ and its derivatives in $\xi$ are bounded on
compacts (by $\|u_n\|_{H^{-s'}}\|\partial^\alpha_\xi(\varphi e_{-\xi})\|_{H^{s'}}
\le C_\alpha(1+|\xi|)^{s'+|\alpha|}$), so by Arzel\`a--Ascoli a subsequence of
$(\widehat u_n)$ converges locally uniformly, and the splitting
$\|u_n-u_{n'}\|^2_{H^{-s}}\le\int_{|\xi|\le R}+\;2\sup_n\|u_n\|^2_{H^{-s'}}
(1+R^2)^{-(s-s')}$ shows it is Cauchy in $H^{-s}$.  By Markov's inequality the sets
$\{\|\psi\XN\|_{H^{-s'}}\le R\}$ then provide tightness of $(\psi\XN)_N$ in
$H^{-s}(\R^d)$.

Finally, let $(\psi_j)_{j\ge1}\subset C_c^\infty$ with $\psi_j\equiv1$ on the ball
$B_j$.  Given $\varepsilon>0$ choose $R_j$ with
$\sup_N\PP(\|\psi_j\XN\|_{H^{-s'}}>R_j)\le\varepsilon2^{-j}$.  The set
$K_\varepsilon:=\{u:\|\psi_ju\|_{H^{-s'}}\le R_j\ \forall j\}$ is relatively compact in
$H^{-s}_{\mathrm{loc}}(\R^d)$: by the previous paragraph and a diagonal extraction,
every sequence in $K_\varepsilon$ has a subsequence such that $\psi_ju_n$ converges in
$H^{-s}$ for every $j$, and for any $h\in C_c^\infty$ one has $h=h\psi_j$ for
$j$ large, with multiplication by $h$ bounded on $H^{-s}$.  Since
$\PP(\XN\notin K_\varepsilon)\le\varepsilon$ uniformly in $N$, the family $(\XN)$ is
tight in the Polish space $H^{-s}_{\mathrm{loc}}(\R^d)$.  By Prokhorov's theorem and
Step~2---the pairings $u\mapsto\la u,f\ra$, $f\in C_c^\infty$, being continuous on
$H^{-s}_{\mathrm{loc}}$ and law-determining---every subsequential limit is the
Gaussian field $\mathcal X^P$ with covariance \eqref{eq:ws-cov-limit}, which proves
convergence in law in $H^{-s}_{\mathrm{loc}}(\R^d)$
\cite{billingsley_convergence_1999,mitoma_tightness_1983}.
\end{proof}

\section{Duality: macroscopic derivatives are microscopic differences}\label{sec:duality}

In the applications of Section~\ref{sec:applications} the field is defined on all of $\Z^d$ (not periodized) and the natural test functions are $f\in C_c^\infty(\R^d)$; the statement of the theory in that setting is the whole-space Theorem~\ref{thm:wholespace}.  Here we prove the duality between the macroscopic differential observable $\XN$ and the microscopic difference operator $P_N$ (Proposition~\ref{prop:duality}); this is the form actually used in the finite-susceptibility applications, where $\XN$ itself degenerates for $m\ge1$.

Before the formal statement, here is the mechanism in the simplest case, $P=\partial_v=v\cdot\nabla$
(order $m=1$).  The macroscopic observable $\XN(f)$ of \eqref{eq:def:PN} (in the
whole-space form \eqref{eq:def:PN-ws}) tests the field
against the \emph{differentiated} test function $(\partial_v f)(x/N)$.  Discrete summation by
parts moves that derivative off the test function and onto the field: up to a vanishing
remainder, $\XN(f)$ is, then, the fluctuation field of the microscopic difference
$\nabla_N\phi$ of the field, tested against $f$.  Differences may thus be placed on whichever
side is convenient---on $f$ they generate the symbol $p$, on $\phi$ they generate the
microscopic derivative whose fluctuations are the physical object (for $P=\partial_v$, the
gradient of a $\nabla\varphi$-interface model).  Proposition~\ref{prop:duality} makes this
precise and fixes the correct normalizations.  The mechanism is well
known---it is already explicit in Naddaf and Spencer
\cite{naddaf_homogenization_1997}---but its details at this level of generality do not
seem to have been written down.  For an $N$-independent field of finite susceptibility,
the differentiated microscopic field $P_N\phi$ carries the central-limit normalization
$N^{m-d/2}$, and the resulting observable $\YN$ equals $(-1)^m$ times the order-zero
fluctuation field $\WN$ tested against $Pf$.  The general order-$m$ case is the same
computation with $m$ successive differences.

\begin{proposition}[Macroscopic--microscopic duality]\label{prop:duality}
Let $P_N$ be the discrete operator \eqref{eq:def:PNdisc} on $\Z^d$, and let $P_N^*$ be its
adjoint with respect to the counting measure $\sum_{x\in\Z^d}\delta_x$, i.e. the operator
determined by $\sum_x(P_Nu)_x\,v_x=\sum_x u_x\,(P_N^*v)_x$ for finitely supported $u,v$.
Since $\sum_x(\nabla_{N,j}u)_x v_x=-\sum_x u_x(\nabla_{N,j}^*v)_x$ with the backward
difference $(\nabla_{N,j}^*v)_x:=v_x-v_{x-e_j}$, one has
$P_N^*=(-1)^m\sum_{|\alpha|=m}c_\alpha\,\nabla_N^{*\alpha}$ with
$\nabla_N^{*\alpha}:=\prod_j(\nabla_{N,j}^*)^{\alpha_j}$, and discrete symbol
\begin{equation}\label{eq:adjoint-symbol}
\pi_N^*(\theta)=(-1)^m\sum_{|\alpha|=m}c_\alpha\prod_j(1-e^{-2\pi i\theta_j})^{\alpha_j},
\end{equation}
to be distinguished from the symbol
$\pi_N(\theta)=\sum_{|\alpha|=m}c_\alpha\prod_j(e^{2\pi i\theta_j}-1)^{\alpha_j}$ of $P_N$
itself.  Both satisfy, uniformly on compacts in $\theta$,
\begin{equation}\label{eq:symbol-expansion}
\pi_N(\theta/N)=N^{-m}\,p(\theta)+O(N^{-m-1}),\qquad
\pi_N^*(\theta/N)=(-1)^m N^{-m}\,p(\theta)+O(N^{-m-1}),\qquad N\to\infty.
\end{equation}
Write $\WN(g):=N^{-d/2}\sum_{x\in\Z^d}\phi_x\,g(x/N)$ for the order-zero fluctuation
observable.  Then, for every $f\in C_c^\infty(\R^d)$ and every centered stationary field
$(\phi_x)$ with $\sum_z|\Gamma(z)|<\infty$,
\begin{equation}\label{eq:duality-L2}
\YN(f):=N^{m-d/2}\sum_{x\in\Z^d}(P_N\phi)_x\,f(x/N)
=(-1)^m\,\WN(Pf)+o_{L^2}(1).
\end{equation}
\end{proposition}

\begin{proof}
Each first difference has symbol $e^{2\pi i\theta_j}-1=2\pi i\theta_j+O(\theta_j^2)$; at
$\theta/N$ this is $2\pi i\theta_j/N+O(N^{-2})$, so a product of $m$ of them is
$N^{-m}\prod_j(2\pi i\theta_j)^{\alpha_j}+O(N^{-m-1})$; summing over $\alpha$ gives the
expansion of $\pi_N$ in \eqref{eq:symbol-expansion}, and the same computation for the
backward differences $1-e^{-2\pi i\theta_j}=2\pi i\theta_j+O(\theta_j^2)$ gives that of
$\pi_N^*$.  By exact summation by parts on $\Z^d$ (no boundary terms, since $f$ has compact
support),
\[
N^{m-d/2}\sum_x(P_N\phi)_x f(x/N)=N^{m-d/2}\sum_x\phi_x\,(P_N^*f)(x/N),
\]
with
\[
(P_N^*f)(x/N)=(-1)^m\sum_{|\alpha|=m}c_\alpha(\nabla_N^{*\alpha}f)(x/N).
\]
The backward differences sample the nearby points $(x-j)/N$, $0\le j_i\le m$, so
$(P_N^*f)(x/N)$ may be nonzero slightly \emph{outside} $N\supp f$; fix the compact
neighbourhood $K:=\{t\in\R^d:\operatorname{dist}(t,\supp f)\le1\}$, which contains all
sample points for $N\ge m$.  Taylor expansion of the smooth $f$ gives
$(P_N^*f)(x/N)=(-1)^m N^{-m}(Pf)(x/N)+r_N(x)$ with $r_N(x)=O(N^{-m-1})$ uniformly in
$x$, and $r_N(x)=0$ for $x\notin NK$; hence, using $N^{m-d/2}\cdot N^{-m}=N^{-d/2}$,
\begin{equation*}
\begin{split}
\YN(f)& =N^{m-d/2}\sum_x\phi_x\,(P_N^*f)(x/N)\\
& =(-1)^m N^{-d/2}\sum_x\phi_x\,(Pf)(x/N)+R_N =(-1)^m\,\WN(Pf)+R_N,
\end{split}
\end{equation*}
with $R_N=N^{m-d/2}\sum_{x\in NK\cap\Z^d}\phi_x\,r_N(x)$.  Since $r_N=O(N^{-m-1})$,
\begin{equation*}
\begin{split}
\E[R_N^2]&=N^{2m-d}\!\!\sum_{x,y\in NK\cap\Z^d}\!\!\Gamma(x-y)\,r_N(x)r_N(y)\\
&\le C\,N^{2m-d}\,N^{-2m-2}\!\!\sum_{x,y\in NK\cap\Z^d}\!\!|\Gamma(x-y)|\\
&\le C'N^{-2}\sum_z|\Gamma(z)|\to0,
\end{split}
\end{equation*}
since $\#(NK\cap\Z^d)\asymp N^d$.  This proves \eqref{eq:duality-L2}.
\end{proof}

\begin{remark}\label{rem:variable}
The duality with a microscopic difference operator relies on $P$ having constant coefficients.  For fixed variable (macroscopic) coefficients $P=\sum_{|\alpha|=m}c_\alpha(x)\partial^\alpha$
the symbol depends on $x$ and the covariance no longer factorizes through a single Fourier multiplier; one then expects a variable-coefficient limiting covariance. A constant effective operator arises only in a genuinely different, homogenization regime, driven by rapidly oscillating or random microscopic coefficients, as in the $\nabla\varphi$ literature \cite{naddaf_homogenization_1997,armstrong_c2_2022,armstrong_quantitative_2024}.  We do not pursue either direction here.
\end{remark}

\section{Two criteria for the cumulant hypothesis}\label{sec:applications}

Once the covariance scaling and spectral bounds required by
Theorems~\ref{thm:main} and~\ref{thm:wholespace} have been verified,
the remaining model-specific input is the higher-cumulant hypothesis.
This section gives two criteria for it: a tree--graph bound on the
cumulants of the field (\S\ref{sec:treegraph}), which also yields an explicit decay rate
for the cumulants of the fluctuation field, and the Dobrushin uniqueness condition for
Gibbs fields (\S\ref{sec:dobrushin-tree-graph}), which implies the tree--graph bound for
every bounded single-site observable.

\subsection{Tree--graph bounds imply vanishing cumulants}\label{sec:treegraph}

We give a checkable sufficient condition for Assumption~\ref{ass:cumulants}.  The pointwise bound below is a convenient tree-sum form of the defining estimate of F\'eray's theory of \emph{weighted dependency graphs} \cite{feray_weighted_2018}: F\'eray's bound is formulated with the maximal spanning-tree weight, and since $\max_Tw(T)\le\sum_Tw(T)\le r^{r-2}\max_Tw(T)$, the two formulations are equivalent up to $r$-dependent constants, so \eqref{eq:tree-graph} is implied by the standard weighted-dependency-graph bound of \cite{feray_weighted_2018,dousse_weighted_2019}.
Section~\ref{sec:dobrushin-tree-graph} verifies it for Gibbs fields in the Dobrushin
regime with a polynomial weight; exponential and covariance weights are covered alike.  With
the covariance weight it is not satisfied by non-linear functionals of Gaussian fields,
although Hermite polynomial functionals do satisfy Assumption~\ref{ass:cumulants} under the
Breuer--Major summability condition (Remark~\ref{rem:nonlinear-gaussian}).

\begin{definition}[Tree--graph bound]\label{def:treegraph}
Let $(X_i)_{i\in I}$ be centered with finite moments and let $w:I\times I\to[0,\infty)$ be
symmetric.  We say $(X_i)$ satisfies the \emph{tree--graph bound with weight $w$} if for
every $r\ge2$ there is $D_r<\infty$ such that, for all $i_1,\dots,i_r$ (not necessarily
distinct),
\begin{equation}\label{eq:tree-graph}
\big|\cum(X_{i_1},\dots,X_{i_r})\big|\le D_r\sum_{T\in\mathcal T_r}\prod_{(u,v)\in E(T)}w(i_u,i_v),
\end{equation}
$\mathcal T_r$ being the set of spanning trees on $\{1,\dots,r\}$.
\end{definition}

\begin{proposition}[Tree--graph bounds imply the cumulant hypothesis]\label{prop:tree-graph}
Let $(\phi^N_x)_{x\in\TN}$ be centered and stationary, satisfying \eqref{eq:tree-graph}
with translation-invariant weight $w(x,y)=w_N(x-y)$, where $w_N:\TN\to[0,\infty)$ (an
$N$-periodic function), $\|w_N\|_{\ell^1}:=\sum_{z\in\TN}w_N(z)$, and $D_r$ is independent of $N$.
Then, for every $r\ge2$,
\begin{equation}\label{eq:cumulant-bound}
\sum_{x_1,\dots,x_r\in\LN}\big|\kappa_r^{(N)}(x_1,\dots,x_r)\big|\le D_r\,r^{r-2}\,|\LN|\,\|w_N\|_{\ell^1}^{\,r-1}.
\end{equation}
Consequently, for every scalar constant-coefficient operator $P$ of order $m\ge0$ and all
$f_1,\dots,f_r\in C^\infty(\T^d)$,
\begin{equation}\label{eq:cum-linear}
\big|\cum(\XN(f_1),\dots,\XN(f_r))\big|
\le D_r\,r^{r-2}\,N^{-r(d/2+m)}\,|\LN|\,\|w_N\|_{\ell^1}^{\,r-1}\prod_{j=1}^r\|Pf_j\|_\infty.
\end{equation}
In particular, if $\sup_N\|w_N\|_{\ell^1}<\infty$ then Assumption~\ref{ass:cumulants} holds for
every $m\ge0$ and every $d\ge1$.
\end{proposition}

\begin{proof}
Fix $r\ge2$.  By \eqref{eq:tree-graph} and translation invariance,
\[
\sum_{x_1,\dots,x_r\in\LN}\big|\kappa_r^{(N)}\big|
\le D_r\sum_{T\in\mathcal T_r}\sum_{x_1,\dots,x_r\in\LN}\prod_{(u,v)\in E(T)}w_N(x_u-x_v).
\]
Fix a tree $T$, pick a leaf $v$ with neighbour $u$; summing first over $x_v$ (the map
$x_v\mapsto x_u-x_v$ being a bijection of $\LN$ onto $\TN$),
$\sum_{x_v\in\LN}w_N(x_u-x_v)=\sum_{z\in\TN}w_N(z)=\|w_N\|_{\ell^1}$.  Removing $v$ leaves a
tree on $r-1$ vertices; iterating $r-1$ times,
$\sum_{x_\bullet}\prod_{(u,v)}w_N(x_u-x_v)\le|\LN|\,\|w_N\|_{\ell^1}^{r-1}$.  Cayley's formula
$|\mathcal T_r|=r^{r-2}$ gives \eqref{eq:cumulant-bound}.  Then \eqref{eq:cum-linear}
follows by multilinearity and bounding $|Pf_j|\le\|Pf_j\|_\infty$.  Finally, with
$|\LN|=N^d$ and $\sup_N\|w_N\|_{\ell^1}<\infty$,
\[
\big|\cum(\XN(f_1),\dots,\XN(f_r))\big|\lesssim N^{d-r(d/2+m)}\prod_j\|Pf_j\|_\infty\to0
\quad(r\ge3),
\]
since for $r\ge3$ and $m\ge0$ we have $r(d/2+m)\ge 3d/2>d$, so the exponent
$d-r(d/2+m)<0$.  This is Assumption~\ref{ass:cumulants}.
\end{proof}

\begin{remark}[Non-linear functionals of Gaussian fields]\label{rem:nonlinear-gaussian}
The criterion is sufficient but not necessary for Assumption~\ref{ass:cumulants}, and
non-linear functionals of Gaussian fields are the instructive counterexample.  Let
$(\phi_x)_{x\in\Z^d}$ be a centered stationary Gaussian field with $\rho(x):=\E[\phi_0\phi_x]$,
$\rho(0)=1$.  Recall that $F\in L^2(\mathcal N(0,1))$ has \emph{Hermite rank} $q$ if
$F=\sum_{\ell\ge q}c_\ell\,\mathrm{He}_\ell$ with $c_q\ne0$, $(\mathrm{He}_\ell)$ being the Hermite
polynomials, and that $F$ is a \emph{Hermite polynomial functional} if this sum is finite.  Let
$Y_x:=\phi_x^2-1=\mathrm{He}_2(\phi_x)$, the Hermite polynomial functional of rank $2$.  The triangle
diagram gives $\cum(Y_0,Y_x,Y_y)=8\,\rho(x)\rho(y)\rho(x-y)$, which no constant $D_3$
dominates pointwise by $\sum_T\prod_{(i,j)\in E(T)}|\rho|^2$ (with all three correlations
equal to $t\to0$ the diagram is $\sim t^3$ while each spanning tree is $\sim t^4$), so the
pointwise bound \eqref{eq:tree-graph} with weight $|\rho|^2$ fails.  The summed form
\eqref{eq:cumulant-bound} fails as well, even in the summable regime: for
$\rho(z)\asymp|z|^{-3/5}$ in $d=1$ one has $\sum_z|\rho(z)|^2<\infty$, yet the summed
triangle contribution is $\sum_{x,y,z\in\LN}|\rho(x-y)\rho(y-z)\rho(x-z)|\asymp N^{6/5}$,
which exceeds $|\LN|\,\|\rho^2\|_{\ell^1}^2\asymp N$.  Nevertheless
$N^{-3/2}\sum_{x,y,z\in\LN}|\cum(Y_x,Y_y,Y_z)|\asymp N^{-3/10}\to0$, and in general
the cumulant condition \eqref{eq:ws-cum} holds (with $m=0$) for $Y_x=F(\phi_x)$ with $F$ any Hermite
polynomial functional of rank $q$ (so that $Y_0$ has finite moments of all orders and the joint
cumulants are finite sums of diagrams) as soon as $\sum_z|\rho(z)|^q<\infty$: this is the connected-diagram estimate at the core
of the Breuer--Major theorem \cite{breuer_central_1983}, see \cite[Ch.~7]{nourdin_normal_2012}
and, for the field-level form on $\Z^d$ and its multivariate version,
\cite{coppini_central_2025}. Theorem~\ref{thm:wholespace} therefore applies to such
functionals---for instance to the centered squared gradient of the Gaussian free field on $\Z^d$, $d\ge3$,
\cite{cipriani_properties_2023}, the multivariate analogue of a rank-two Hermite polynomial
functional (a quadratic form in the correlated Gaussian vector $\nabla\phi_x$; see
\cite{coppini_central_2025} for the multivariate form) with absolutely summable covariance,
whose volume average converges to white noise---but through the diagram estimate rather
than through the criterion of this section; we do not develop this direction here.
\end{remark}

\begin{remark}\label{rem:treegraph-Zd}
The case $r=2$ in \eqref{eq:cumulant-bound} is included for use in
the variance estimates and is consistent with the finite-susceptibility hypothesis
$\sup_N\|w_N\|_{\ell^1}<\infty$.  The criterion transfers verbatim to the whole-space setting: for an
$N$-independent field on $\Z^d$ satisfying \eqref{eq:tree-graph} with weight
$w(x,y)=\bar w(x-y)$ and $\|\bar w\|_{\ell^1}:=\sum_{z\in\Z^d}\bar w(z)<\infty$, the same leaf-peeling
argument bounds the sum restricted to $\{x_j/N\in K\}$, $K\subset\R^d$ compact, by
$D_r\,r^{r-2}\,\#(NK\cap\Z^d)\,\|\bar w\|_{\ell^1}^{\,r-1}=O(N^d)$, which yields the cumulant condition \eqref{eq:ws-cum} of
Theorem~\ref{thm:wholespace}; this is the form used in \S\ref{sec:dobrushin-tree-graph}.
\end{remark}

\begin{remark}[Berry--Esseen rates from the tree--graph bound]\label{rem:berry-esseen}
The bound \eqref{eq:cum-linear} is quantitative.  For the order-zero observable
$\mathcal X_N^{\mathrm{Id}}(g)=N^{-d/2}\sum_{x\in\LN}\phi^N_x\,g(x/N)$ it reads
\[
|\cum_r(\mathcal X_N^{\mathrm{Id}}(g))|\le D_r\,r^{r-2}\,\big(\sup_N\|w_N\|_{\ell^1}\big)^{r-1}\|g\|_\infty^{r}
\,N^{-(r-2)d/2}
\]
for $r\ge2$.  If the constants grow at most factorially,
$D_r\le(r!)^{\delta}A^{r}$ for some $\delta\ge0$ and $A<\infty$, and
$\Var(\mathcal X_N^{\mathrm{Id}}(g))\to\sigma^2>0$, this is the Statulevi\v{c}ius condition on the cumulants of
the standardized variable with parameter of order $N^{d/2}$, and the Saulis--Statulevi\v{c}ius
lemma \cite[Lemma~2.3, Cor.~2.1]{saulis_limit_1991} bounds the Kolmogorov distance between
$\mathcal X_N^{\mathrm{Id}}(g)/\sqrt{\Var(\mathcal X_N^{\mathrm{Id}}(g))}$ and $\mathcal N(0,1)$ by $C\,N^{-d/(2(1+2\delta))}$; for
$\delta=0$ (e.g.\ uniformly bounded finite-range fields with the standard dependency-graph
cumulant bound) this is the optimal $N^{-d/2}=|\LN|^{-1/2}$.  The constants obtained from the Dobrushin argument are not optimized. Tracking their dependence on the cumulant order gives factorial-power bounds, which can be used to obtain quantitative estimates when the limiting variance is positive. To the authors' knowledge, a rate for the random field, uniform over a metric on $H^{-s}$, is open.
\end{remark}

\subsection{Dobrushin uniqueness implies tree--graph bounds}
\label{sec:dobrushin-tree-graph}

We now give a model-independent sufficient condition for the tree--graph bound of
Definition~\ref{def:treegraph}.  For finite-range Gibbs specifications, the Dobrushin
uniqueness condition implies that bound for every bounded single-site observable, with one
fixed summable weight; Proposition~\ref{prop:tree-graph} then yields the cumulant condition
\eqref{eq:ws-cum}.  We state the resulting fluctuation theorem once and list the models to
which it applies.

Let $\mathsf S$ be a standard Borel single-spin space and let
$\mathcal G=(\mathcal G_x)_{x\in\Z^d}$ be a translation-invariant, quasilocal specification
on $\mathsf S^{\Z^d}$ which admits a Gibbs measure.  Its Dobrushin interdependence
coefficients are for $x\neq y$
\begin{equation}
 c_{xy}:=\sup\big\{\,|\mathcal G_x(A\mid\eta)-\mathcal G_x(A\mid\eta')|:\
 A\subset\mathsf S\ \text{measurable},\ \eta,\eta'\in\mathsf S^{\Z^d},\
 \eta_z=\eta'_z\ \text{for }z\neq y\big\},
 \label{eq:dobrushin-coefficients-new}
\end{equation}
and $c_{xx}:=0$.  We assume that the interdependence matrix has finite range: there is
$R<\infty$ such that $c_{xy}=0$ whenever $|x-y|_1>R$.  The Dobrushin condition is
\begin{equation}
 c(\mathcal G):=\sup_{x\in\Z^d}\sum_{y\in\Z^d}c_{xy}<1.
 \label{eq:dobrushin-condition-new}
\end{equation}
It implies uniqueness of the Gibbs measure
\cite[Theorem 6.31]{friedli_velenik_2017}; 
we denote the unique state by $\mu$
and write $\E_\mu$, $\Cov_\mu$, $\cum_\mu$ for expectation, covariance and joint cumulants
under $\mu$.  By translation invariance of $\mathcal G$ and uniqueness, $\mu$ is
translation invariant.

\begin{proposition}[Dobrushin uniqueness implies the tree--graph bound]
\label{prop:dobrushin-implies-tree-graph}
Assume \eqref{eq:dobrushin-condition-new} and define the translation-invariant summable weight
\begin{equation}
 w(x,y):=(1+|x-y|_1)^{-(d+1)},
 \qquad
 \|w\|_{\ell^1}:=\sum_{z\in\Z^d}(1+|z|_1)^{-(d+1)}<\infty.
 \label{eq:fixed-polynomial-weight}
\end{equation}
For every $r\geq2$ there is a constant $D_r<\infty$, depending only on $r$ and on the
Dobrushin matrix, such that for all sites $x_1,\ldots,x_r\in\Z^d$ (not necessarily
distinct) and all measurable $F_1,\ldots,F_r:\mathsf S\to\R$ with $\|F_i\|_\infty\leq1$,
\begin{equation}
 \big|\cum_\mu\big(F_1(\sigma_{x_1}),\ldots,F_r(\sigma_{x_r})\big)\big|
 \leq D_r\sum_{T\in\mathcal T_r}\prod_{\{i,j\}\in E(T)}w(x_i,x_j).
 \label{eq:dobrushin-tree-graph-bound}
\end{equation}
By multilinearity the bound holds for arbitrary bounded $F_i$ with the constant
$D_r\prod_i\|F_i\|_\infty$.  Thus every bounded single-site observable satisfies
Definition~\ref{def:treegraph} with the same weight $w$ for all cumulant orders, and so does
any finite collection of bounded single-site observables simultaneously.  Consequently, for
every finite $\Lambda\subset\Z^d$,
\begin{equation}
 \sum_{x_1,\ldots,x_r\in\Lambda}
 \big|\cum_\mu\big(F_1(\sigma_{x_1}),\ldots,F_r(\sigma_{x_r})\big)\big|
 \leq D_r\,r^{r-2}\,|\Lambda|\,\|w\|_{\ell^1}^{\,r-1},
 \label{eq:dobrushin-volume-cumulant-bound}
\end{equation}
and the whole-space cumulant condition \eqref{eq:ws-cum} holds for every $r\geq3$ and every
differential order $m\geq0$.
\end{proposition}

\begin{proof}
Write $\mathbf C=(c_{xy})_{x,y\in\Z^d}$ and
\[
 \mathbf D:=\sum_{j=0}^{\infty}\mathbf C^j=(I-\mathbf C)^{-1}.
\]
For a bounded local function $U$ on $\mathsf S^{\Z^d}$, let
\[
 \delta_x(U):=
 \sup_{\substack{\eta,\eta'\\
                  \eta_z=\eta'_z\ \text{for }z\neq x}}
 |U(\eta)-U(\eta')|.
\]
The Dobrushin--F\"ollmer covariance estimate gives
\begin{equation}
 |\Cov_\mu(U,V)|
 \leq
 \frac14\sum_{x,y\in\Z^d}
 \delta_x(U)\,\mathbf D_{xy}\,\delta_y(V);
 \label{eq:follmer-covariance-new}
\end{equation}
see \cite{Follmer1982} and
\cite[Theorem~2.2]{SalasSokal1997}.

Since $\mathbf C$ has range $R$, $(\mathbf C^j)_{xy}=0$ when
$j<|x-y|_1/R$.  Moreover,
\[
 \sup_x\sum_y(\mathbf C^j)_{xy}\leq c(\mathcal G)^j.
\]
For $0<c(\mathcal G)<1$ it follows that
\begin{equation}
 \mathbf D_{xy}
 \leq
 \sum_{j\geq\lceil |x-y|_1/R\rceil}c(\mathcal G)^j
 \leq
 \frac{1}{1-c(\mathcal G)}\,e^{-\tau|x-y|_1},
 \qquad
 \tau:=\frac1R\log\frac1{c(\mathcal G)}.
 \label{eq:dobrushin-resolvent-decay-new}
\end{equation}
The case $c(\mathcal G)=0$ is independent and is immediate.  Therefore, if
$U$ and $V$ are bounded local functions supported in finite sets $A$ and
$A'$, then
\begin{equation}
 |\Cov_\mu(U,V)|
 \leq C_{A,A'}\|U\|_\infty\|V\|_\infty
 e^{-\tau\operatorname{dist}(A,A')},
 \label{eq:dobrushin-two-block-new}
\end{equation}
and when $|A|,|A'|\leq r$ the prefactor is bounded by a constant depending only
on $r$ and the Dobrushin matrix.

Set $\bar F_i:=F_i-\E_\mu[F_i(\sigma_0)]$; by translation invariance of $\mu$,
$\bar F_i(\sigma_{x})$ is centered for every $x$, and $\|\bar F_i\|_\infty\le2$.
Cumulants of order at least two are invariant under adding constants, so it
is enough to bound $\cum_\mu(\bar F_1(\sigma_{x_1}),\ldots,\bar F_r(\sigma_{x_r}))$.  Fix a
nontrivial partition $I\sqcup J=\{1,\ldots,r\}$ and put
\[
 \operatorname{dist}(x_I,x_J):=\min_{i\in I,\,j\in J}|x_i-x_j|_1.
\]
For $B\subset\{1,\ldots,r\}$, write
\[
 \mathfrak m_B:=\E_\mu\Big[\prod_{i\in B}\bar F_i(\sigma_{x_i})\Big],
 \qquad
 \widetilde{\mathfrak m}_B:=\mathfrak m_{B\cap I}\,\mathfrak m_{B\cap J},
 \qquad \mathfrak m_\varnothing:=1.
\]
If $B$ meets both $I$ and $J$, then \eqref{eq:dobrushin-two-block-new},
applied to the two products, gives
\begin{equation}
 |\mathfrak m_B-\widetilde{\mathfrak m}_B|
 \leq C_r\,e^{-\tau\operatorname{dist}(x_I,x_J)}.
 \label{eq:moment-splitting-new}
\end{equation}
If $B$ lies in one side of the partition, the difference is zero.

The moment--cumulant formula is
\begin{equation}
 \cum_\mu\big(\bar F_1(\sigma_{x_1}),\ldots,\bar F_r(\sigma_{x_r})\big)
 =
 \sum_{\pi\in\mathcal P_r}
 (|\pi|-1)!(-1)^{|\pi|-1}
 \prod_{B\in\pi}\mathfrak m_B.
 \label{eq:moment-cumulant-new}
\end{equation}
The same polynomial evaluated at the moments $\widetilde{\mathfrak m}_B$ vanishes,
because these are the joint moments obtained by replacing the two subvectors
indexed by $I$ and $J$ by independent copies.  Subtracting the two
polynomials, telescoping each product, and using
\eqref{eq:moment-splitting-new} and boundedness of all moments yields
\begin{equation}
 \big|\cum_\mu\big(\bar F_1(\sigma_{x_1}),\ldots,\bar F_r(\sigma_{x_r})\big)\big|
 \leq C_r\,e^{-\tau\operatorname{dist}(x_I,x_J)}
 \label{eq:cumulant-cut-new}
\end{equation}
for every nontrivial partition $I\sqcup J$, with a new constant $C_r$.

Let
\[
 \mst(x_1,\ldots,x_r)
 :=\min_{T\in\mathcal T_r}
 \sum_{\{i,j\}\in E(T)}|x_i-x_j|_1
\]
be the minimum spanning-tree length, and let $T_*$ attain this minimum.
Remove a longest edge of $T_*$, of length $L_*$.  The resulting components
define a partition $I\sqcup J$.  By the cut property of minimum spanning
trees, every edge crossing this cut has length at least $L_*$; hence
$\operatorname{dist}(x_I,x_J)\geq L_*$.  Since $T_*$ has $r-1$ edges,
\[
 L_*\geq\frac{\mst(x_1,\ldots,x_r)}{r-1}.
\]
Using \eqref{eq:cumulant-cut-new},
\begin{equation}
 \big|\cum_\mu\big(\bar F_1(\sigma_{x_1}),\ldots,\bar F_r(\sigma_{x_r})\big)\big|
 \leq
 C_r
 \exp\!\left\{-\frac{\tau}{r-1}\mst(x_1,\ldots,x_r)\right\}.
 \label{eq:dobrushin-MST-exponential-new}
\end{equation}

The decay rate in \eqref{eq:dobrushin-MST-exponential-new} depends on $r$.
This does not prevent a tree--graph estimate with one fixed summable weight.
Indeed, let
\[
 C'_r:=
 \sup_{t\geq0}e^{-\tau t/(r-1)}(1+t)^{d+1}<\infty.
\]
Applying
$e^{-\tau t/(r-1)}\leq C'_r(1+t)^{-(d+1)}$ to every edge of $T_*$ gives
\[
 \exp\!\left\{-\frac{\tau}{r-1}\mst(x_1,\ldots,x_r)\right\}
 \leq
 (C'_r)^{r-1}
 \prod_{\{i,j\}\in E(T_*)}w(x_i,x_j).
\]
The last product is bounded by the sum over all spanning trees.  This proves
\eqref{eq:dobrushin-tree-graph-bound}, with
$D_r:=C_r\,(C'_r)^{r-1}$.  Equation
\eqref{eq:dobrushin-volume-cumulant-bound} is then Proposition~\ref{prop:tree-graph}.  If
$|\Lambda_N|=O(N^d)$, its right-hand side is $O_r(N^d)$, and therefore
\[
 N^{-r(d/2+m)}O_r(N^d)
 =O_r\!\left(N^{d-r(d/2+m)}\right)\longrightarrow0,
 \qquad r\geq3,
\]
which is Assumption~\ref{ass:cumulants}.  The same argument applies to the sums over
$NK\cap\Z^d$, which gives \eqref{eq:ws-cum}.
\end{proof}

\begin{remark}
\label{rem:dobrushin-strength-converse}
There is no converse from the tree--graph bound, or from Assumption~\ref{ass:cumulants},
to Dobrushin uniqueness.  These cumulant properties concern one selected
Gibbs state, whereas the Dobrushin condition is a uniform contraction
property of the specification and implies uniqueness.  At zero field and
sufficiently low Ising temperature, the plus and minus states are distinct,
so Dobrushin uniqueness fails; nevertheless, the low-temperature plus state
satisfies an exponential weighted-dependency-graph bound in the
contour-expansion regime \cite{dousse_weighted_2019}.
\end{remark}

The model-independent fluctuation statement can now be recorded once.

\begin{corollary}[Gaussian fluctuations in the Dobrushin regime]
\label{cor:dobrushin-fluctuation-fields}
Assume \eqref{eq:dobrushin-condition-new}.  Let $F^1,\ldots,F^n:\mathsf S\to\R$ be bounded
and put
\[
 \phi_x^a:=F^a(\sigma_x)-\E_\mu[F^a(\sigma_x)],
 \qquad a=1,\ldots,n.
\]
Their covariances are exponentially summable by \eqref{eq:dobrushin-two-block-new}.  Hence
the susceptibility matrix
\begin{equation}
 \boldsymbol\chi_{ab}
 :=\sum_{z\in\Z^d}
 \Cov_\mu(\phi_0^a,\phi_z^b)
 \label{eq:dobrushin-susceptibility-matrix-new}
\end{equation}
is finite and positive semidefinite.  For
$f\in C_c^\infty(\R^d)$, the joint field
\[
 \WN^a(f):=N^{-d/2}\sum_{x\in\Z^d}
 \phi_x^a f(x/N)
\]
converges, in the sense of Theorem~\ref{thm:wholespace}, to an
$\R^n$-valued white noise $\mathcal W$ with
\begin{equation}
 \E[\mathcal W^a(f)\mathcal W^b(g)]
 =\boldsymbol\chi_{ab}\la f,g\ra_{L^2(\R^d)}.
 \label{eq:dobrushin-vector-white-noise-new}
\end{equation}
The limit may be degenerate when $\boldsymbol\chi$ has a nontrivial kernel.

Let $P$ be a homogeneous constant-coefficient operator of order $m$, and let
$P_N$ be the associated lattice difference operator \eqref{eq:def:PNdisc}.  Then, jointly in
$a=1,\ldots,n$,
\begin{equation}
 N^{m-d/2}\sum_{x\in\Z^d}
 (P_N\phi^a)_x f(x/N)
 \xrightarrow[]{\ d\ }
 (-1)^m\mathcal W^a(Pf),
 \label{eq:dobrushin-differentiated-new}
\end{equation}
with limiting covariance
$\boldsymbol\chi_{ab}\la Pf,Pg\ra_{L^2}$.
\end{corollary}

\paragraph{Models in the Dobrushin regime.}
Corollary~\ref{cor:dobrushin-fluctuation-fields} applies, with no further cumulant
argument, wherever the Dobrushin condition \eqref{eq:dobrushin-condition-new} has been
verified; the classical instances are the following.
\begin{itemize}[leftmargin=*]
\item The nearest-neighbour ferromagnetic Ising model on $\Z^d$ at high temperature,
$d\tanh(2\beta J)<1$ \cite{dobrushin_description_1968,georgii_gibbs_2011}; the unique state
is spin-flip invariant, so the spin field $\sigma_x$ is centered and
Corollary~\ref{cor:dobrushin-fluctuation-fields} with $F(s)=s$ gives the magnetization
white-noise limit and every differentiated limit, with intensity the susceptibility
$\sum_z\E_\mu[\sigma_0\sigma_z]\in(0,\infty)$.
\item The $q$-state Potts model, ferromagnetic or antiferromagnetic, and more generally
every finite-state, finite-range model with bounded pair interactions, at sufficiently high
temperature
\cite[Ch.~8]{georgii_gibbs_2011}; for Potts the colour fields
$\phi^a_x=\mathbf 1_{\{\sigma_x=a\}}-1/q$, $a=1,\ldots,q$, converge jointly; since
$\sum_a\phi^a_x\equiv0$, the vector $(1,\dots,1)$ lies in the kernel of the susceptibility matrix
and the limit is degenerate in that direction.
\item The antiferromagnetic Potts model on $\Z^d$ with $q>4d$ colours, at all temperatures
including zero temperature, i.e.\ uniform proper $q$-colourings \cite{SalasSokal1997}.
\item The hard-core lattice gas on $\Z^d$ at low activity, below $1/(2d-1)$
\cite{Weitz2005,georgii_gibbs_2011}.
\end{itemize}
The Dobrushin condition is sufficient, not necessary
(Remark~\ref{rem:dobrushin-strength-converse}): the same conclusions hold wherever the
tree--graph bound is established by another method.

\section{The ``differentiation removes correlation'' paradigm}\label{sec:paradigm}

In Fourier space, the operator multiplies the spectral density by $|p(\theta)|^2$ and so suppresses its low-frequency mass, which is exactly the mass responsible for long-range correlation. For a field of finite susceptibility this gain is invisible: the macroscopic observable degenerates and the duality of Section~\ref{sec:duality} hands the limit to the microscopic difference field, with the same summability threshold as for $m=0$.  The gain is genuine in the long-memory regime, where the spectral density is singular at the origin.  The following proposition makes it precise at the level of the covariance and the remark after it says what differentiation does not do.

\begin{proposition}[Over-, exact and under-differentiation]\label{prop:trichotomy}
Let $(\phi_x)_{x\in\Z^d}$ be centered and stationary, with finite moments of all orders and
spectral density
\begin{equation}\label{eq:spectral-singularity}
\widehat\Gamma(\theta)=c_0\,|\theta|^{-\nu}\,(1+o(1)),\qquad\theta\to0,
\end{equation}
for some $\nu\in[0,d)$ and $c_0>0$, bounded outside every neighbourhood of the origin.  Let
$P$ have order $m$ with symbol $p$ and let $\XN$ be the whole-space observable
\eqref{eq:def:PN-ws}.  Then, according to how $m$ compares with $\nu$:
\begin{enumerate}[label=\textup{(\roman*)}]

\item \textup{(over-differentiation, $\nu<2m$)} $\Var(\XN(f))\to0$ for every $f$, so
$\XN\to0$ in probability;
\item \textup{(exact match, $\nu=2m$)} the covariance converges,
\begin{equation}\label{eq:spectral-BM-cov}
\Cov\big(\XN(f),\XN(g)\big)\longrightarrow
c_0\int_{\R^d}\overline{\widehat f(\xi)}\,\widehat g(\xi)\,|p(\xi)|^2|\xi|^{-2m}\dd\xi,
\end{equation}
a bounded multiplier homogeneous of degree zero.

\item \textup{(under-differentiation, $\nu>2m$)} $\Var(\XN(f))\asymp N^{\nu-2m}\to\infty$ whenever $\int|p|^2|\widehat f|^2|\xi|^{-\nu}\dd\xi\ne0$, and under the self-similar normalization $N^{-(d+\nu)/2}$  \cite{dobrushin_non-central_1979}, the covariance converges to
\[
c_0\int\overline{\widehat f}\,\widehat g\,|p(\xi)|^2|\xi|^{-\nu}\dd\xi.
\]
For Gaussian $\phi$ the limit law follows: it is $P$ applied to a fractional Gaussian field
$\mathrm{FGF}_{\nu/2}(\R^d)$ in the parametrization of \cite{lodhia_fractional_2016}.
\end{enumerate}
\end{proposition}

\begin{proof}
By \eqref{eq:spectral-singularity}, $N^{-2m}\widehat\Gamma(\xi/N)=c_0\,N^{\nu-2m}|\xi|^{-\nu}(1+o(1))$
for fixed $\xi\ne0$, which gives the pointwise limit of $G_N$ in each of the three cases.
All three cases rest on the same domination: for $\varsigma\in\{\nu,2m\}$, on the box $Q_N$
one has $N^{-\varsigma}\widehat\Gamma(\xi/N)\le C|\xi|^{-\nu}\,N^{\nu-\varsigma}$ for
$|\xi|\le\delta N$ by \eqref{eq:spectral-singularity}, while for $|\xi|>\delta N$ the assumed
boundedness of $\widehat\Gamma$ outside the $\delta$-ball gives
$N^{-\varsigma}\widehat\Gamma(\xi/N)\le C_\delta N^{-\varsigma}\le C'_\delta|\xi|^{-\varsigma}$,
since $|\xi|\le N\sqrt d/2$ on $Q_N$.

In case (ii) ($\varsigma=2m=\nu<d$), we have that
$G_N\le\Phi:=C|\xi|^{-2m}$, which is locally integrable and bounded for $|\xi|\ge1$, so
Assumption~\ref{ass:wholespace}\,\ref{it:ws-scaling}--\ref{it:ws-domination} hold with
$M=0$ and $\gamma=c_0|\cdot|^{-2m}$: the covariance convergence
\eqref{eq:spectral-BM-cov} is Step~1 of the proof of Theorem~\ref{thm:wholespace}, and
under the cumulant condition the full theorem applies.

In case (i) ($\varsigma=2m>\nu$), $G_N\to0$ pointwise with the same locally integrable majorant $C|\xi|^{-\nu}$, so $\Var(\XN(f))\to0$ and Chebyshev gives convergence to $0$ in probability.

In case (iii), the self-similarly rescaled weight $N^{-\nu}\widehat\Gamma(\xi/N)$ satisfies the hypotheses with $\gamma=c_0|\cdot|^{-\nu}$ and majorant $C|\xi|^{-\nu}$ ($\nu<d$), so Step~1 of the proof of Theorem~\ref{thm:wholespace} gives the stated covariance limit,
and hence $\Var(\XN(f))=N^{\nu-2m}\big(c_0\int|p|^2|\widehat f|^2|\xi|^{-\nu}\dd\xi+o(1)\big)$.
For Gaussian $\phi$ the variables $\XN(f)$ are centered Gaussian, so diverging variance
forbids tightness of any subsequence, and the self-similar covariance limit determines
the Gaussian law stated.
\end{proof}

\begin{remark}\label{rem:noncentral}
The proposition is a statement about the covariance and not about Gaussianity. For instance, let $\phi$ be Gaussian and $Y_x:=F(\phi_x)$ with $F$ a Hermite polynomial functional of rank $q\ge2$ (Remark~\ref{rem:nonlinear-gaussian}), in the long-memory regime where the spectral density of $Y$ has the singularity \eqref{eq:spectral-singularity} with $\nu\in(0,d)$.  On the line $\nu=2m$ the normalization $N^{-d/2-m}$ coincides with the non-central normalization of Dobrushin--Major and Taqqu \cite{dobrushin_non-central_1979,taqqu_convergence_1979}. For instance, applying $P$ changes the test function and the spectral multiplier but not the chaos order: the limit of the normalized observable, when nonzero, lies in the $q$th Wiener chaos and is not Gaussian
\cite{janson_gaussian_1997}.
The conditional Gaussian statement in (ii) therefore has content where the cumulant condition is checkable, e.g.\ for Gaussian fields ($q=1$); for
$q\ge2$ the correct conclusion, on the line $\nu=2m$ as well as below it, is the non-central Hermite--Rosenblatt limit, outside the present framework.  Slowly varying corrections to the
correlation and the logarithmic boundary case (susceptibility diverging like $\log N$, normalization $(N^d\log N)^{-1/2}$) are not treated
here but represent an interesting direction of future research. 
\end{remark}

\section*{Acknowledgments}

The authors have used AI tools to review the existing literature, do a coherence verification of the results and do the straightforward steps in the proofs. The authours have double checked and validated every generated output.

\bibliographystyle{plain}
\bibliography{refs}

\begin{thebibliography}{10}

\bibitem{aizenman_sharpness_1987}
Michael Aizenman, David~J. Barsky, and Roberto Fern\'andez.
\newblock The phase transition in a general class of {Ising}-type models is sharp.
\newblock {\em Journal of Statistical Physics}, 47:343--374, 1987.

\bibitem{armstrong_quantitative_2024}
Scott Armstrong and Paul Dario.
\newblock Quantitative hydrodynamic limits of the {Langevin} dynamics for gradient interface models.
\newblock {\em Electronic Journal of Probability}, 29:1--93, 2024.

\bibitem{armstrong_c2_2022}
Scott Armstrong and Wei Wu.
\newblock {$C^2$} regularity of the surface tension for the $\nabla\phi$ interface model.
\newblock {\em Communications on Pure and Applied Mathematics}, 75(2):349--421, 2022.

\bibitem{bierme_central_2011}
Hermine Bierm\'e, Aline Bonami, and Jos\'e~R. Le\'on.
\newblock Central limit theorems and quadratic variations in terms of spectral density.
\newblock {\em Electronic Journal of Probability}, 16:362--395, 2011.

\bibitem{billingsley_convergence_1999}
Patrick Billingsley.
\newblock {\em Convergence of Probability Measures}.
\newblock Wiley, New York, 2nd edition, 1999.

\bibitem{biskup_scaling_2011}
Marek Biskup and Herbert Spohn.
\newblock Scaling limit for a class of gradient fields with nonconvex potentials.
\newblock {\em The Annals of Probability}, 39(1):224--251, 2011.

\bibitem{bolthausen_central_1982}
Erwin Bolthausen.
\newblock On the central limit theorem for stationary mixing random fields.
\newblock {\em The Annals of Probability}, 10(4):1047--1050, 1982.

\bibitem{breuer_central_1983}
P\'eter Breuer and P\'eter Major.
\newblock Central limit theorems for non-linear functionals of {Gaussian} fields.
\newblock {\em Journal of Multivariate Analysis}, 13(3):425--441, 1983.

\bibitem{buchholz_phase_2021}
Simon Buchholz.
\newblock Phase transitions for a class of gradient fields.
\newblock {\em Probability Theory and Related Fields}, 179(3):969--1022, 2021.

\bibitem{cipriani_properties_2023}
Alessandra Cipriani, Rajat~S. Hazra, Alan Rapoport, and Wioletta~M. Ruszel.
\newblock Properties of the gradient squared of the discrete {Gaussian} free field.
\newblock {\em Journal of Statistical Physics}, 190:171, 2023.

\bibitem{conlon_strong_2013}
Joseph~G. Conlon and Thomas Spencer.
\newblock A strong central limit theorem for a class of random surfaces.
\newblock {\em Communications in Mathematical Physics}, 325(1):1--15, 2014.

\bibitem{coppini_central_2025}
Fabio Coppini and Wioletta~M. Ruszel.
\newblock Central limit theorems for non-linear functionals of {Gaussian} fields via {Wiener} chaos decomposition, 2025.
\newblock arXiv:2512.13148.

\bibitem{coppini_wick_2025}
Fabio Coppini, Wioletta~M. Ruszel, and Dirk Schuricht.
\newblock Wick theorem for analytic functions of {Gaussian} fields, 2025.
\newblock arXiv:2507.05131.

\bibitem{cotar_existence_2012}
Codina Cotar and Christof K{\"u}lske.
\newblock Existence of random gradient states.
\newblock {\em The Annals of Applied Probability}, 22(4):1650--1692, 2012.

\bibitem{cramer_theorems_1936}
Harald Cram\'er and Herman Wold.
\newblock Some theorems on distribution functions.
\newblock {\em Journal of the London Mathematical Society}, s1-11(4):290--294, 1936.

\bibitem{deuschel_ray-knight_2024}
Jean-Dominique Deuschel and Pierre-Fran\c{c}ois Rodriguez.
\newblock A {Ray--Knight} theorem for $\nabla\phi$ interface models and scaling limits.
\newblock {\em Probability Theory and Related Fields}, 189(1-2):447--499, 2024.

\bibitem{dobrushin_description_1968}
Roland~L. Dobrushin.
\newblock The description of a random field by means of conditional probabilities and conditions of its regularity.
\newblock {\em Theory of Probability \& Its Applications}, 13(2):197--224, 1968.

\bibitem{dobrushin_non-central_1979}
Roland~L. Dobrushin and P\'eter Major.
\newblock Non-central limit theorems for non-linear functionals of {Gaussian} fields.
\newblock {\em Zeitschrift f\"ur Wahrscheinlichkeitstheorie und verwandte Gebiete}, 50(1):27--52, 1979.

\bibitem{dousse_weighted_2019}
Jehanne Dousse and Valentin F\'eray.
\newblock Weighted dependency graphs and the {Ising} model.
\newblock {\em Annales de l'Institut Henri Poincar\'e D}, 6(4):533--571, 2019.

\bibitem{duminil-copin_new_2016}
Hugo Duminil-Copin and Vincent Tassion.
\newblock A new proof of the sharpness of the phase transition for {Bernoulli} percolation and the {Ising} model.
\newblock {\em Communications in Mathematical Physics}, 343(2):725--745, 2016.

\bibitem{feray_weighted_2018}
Valentin F\'eray.
\newblock Weighted dependency graphs.
\newblock {\em Electronic Journal of Probability}, 23:1--65, 2018.

\bibitem{Follmer1982}
Hans F{\"o}llmer.
\newblock A covariance estimate for {Gibbs} measures.
\newblock {\em Journal of Functional Analysis}, 46(3):387--395, 1982.

\bibitem{friedli_velenik_2017}
Sacha Friedli and Yvan Velenik.
\newblock {\em Statistical Mechanics of Lattice Systems: A Concrete Mathematical Introduction}.
\newblock Cambridge University Press, 2017.

\bibitem{gass_scars_2026}
Louis Gass, Giovanni Peccati, and Michele Stecconi.
\newblock Scars in random waves and the {FGF}$(1/2)$ universality class, 2026.
\newblock arXiv:2606.07842.

\bibitem{georgii_gibbs_2011}
Hans-Otto Georgii.
\newblock {\em Gibbs Measures and Phase Transitions}, volume~9 of {\em De Gruyter Studies in Mathematics}.
\newblock De Gruyter, Berlin, 2nd edition, 2011.

\bibitem{giacomin_equilibrium_2001}
Giambattista Giacomin, Stefano Olla, and Herbert Spohn.
\newblock Equilibrium fluctuations for $\nabla\varphi$ interface model.
\newblock {\em The Annals of Probability}, 29(3):1138--1172, 2001.

\bibitem{giraitis_clt_1985}
Liudas Giraitis and Donatas Surgailis.
\newblock {CLT} and other limit theorems for functionals of {Gaussian} processes.
\newblock {\em Zeitschrift f\"ur Wahrscheinlichkeitstheorie und verwandte Gebiete}, 70(2):191--212, 1985.

\bibitem{janson_gaussian_1997}
Svante Janson.
\newblock {\em Gaussian {Hilbert} Spaces}, volume 129 of {\em Cambridge Tracts in Mathematics}.
\newblock Cambridge University Press, 1997.

\bibitem{kipnis_central_1986}
Claude Kipnis and S.~R.~S. Varadhan.
\newblock Central limit theorem for additive functionals of reversible {Markov} processes and applications to simple exclusions.
\newblock {\em Communications in Mathematical Physics}, 104(1):1--19, 1986.

\bibitem{kotecky_cluster_1986}
R.~Koteck\'y and D.~Preiss.
\newblock Cluster expansion for abstract polymer models.
\newblock {\em Communications in Mathematical Physics}, 103(3):491--498, 1986.

\bibitem{leonenko_limit_1999}
Nikolai Leonenko.
\newblock {\em Limit Theorems for Random Fields with Singular Spectrum}, volume 465 of {\em Mathematics and its Applications}.
\newblock Kluwer, Dordrecht, 1999.

\bibitem{leonenko_tauberian_2013}
Nikolai Leonenko and Andriy Olenko.
\newblock Tauberian and {Abelian} theorems for long-range dependent random fields.
\newblock {\em Methodology and Computing in Applied Probability}, 15(4):715--742, 2013.

\bibitem{lodhia_fractional_2016}
Asad Lodhia, Scott Sheffield, Xin Sun, and Samuel~S. Watson.
\newblock Fractional {Gaussian} fields: a survey.
\newblock {\em Probability Surveys}, 13:1--56, 2016.

\bibitem{mcauley_limit_2026}
Michael McAuley.
\newblock Limit theorems for non-local functionals of smooth {Gaussian} fields via quasi-association, 2026.
\newblock arXiv:2601.04002.

\bibitem{miller_fluctuations_2011}
Jason Miller.
\newblock Fluctuations for the {Ginzburg--Landau} $\nabla\phi$ interface model on a bounded domain.
\newblock {\em Communications in Mathematical Physics}, 308(3):591--639, 2011.

\bibitem{mitoma_tightness_1983}
Itaru Mitoma.
\newblock Tightness of probabilities on $c([0,1];\mathscr s')$ and $d([0,1];\mathscr s')$.
\newblock {\em The Annals of Probability}, 11(4):989--999, 1983.

\bibitem{naddaf_homogenization_1997}
Ali Naddaf and Thomas Spencer.
\newblock On homogenization and scaling limit of some gradient perturbations of a massless free field.
\newblock {\em Communications in Mathematical Physics}, 183(1):55--84, 1997.

\bibitem{newman_normal_1980}
Charles~M. Newman.
\newblock Normal fluctuations and the {FKG} inequalities.
\newblock {\em Communications in Mathematical Physics}, 74(2):119--128, 1980.

\bibitem{newman_gaussian_2018}
Charles~M. Newman and Wei Wu.
\newblock Gaussian fluctuations for the classical {XY} model.
\newblock {\em Annales de l'Institut Henri Poincar\'e, Probabilit\'es et Statistiques}, 54(4):1759--1777, 2018.

\bibitem{nourdin_normal_2012}
Ivan Nourdin and Giovanni Peccati.
\newblock {\em Normal Approximations with {Malliavin} Calculus: From Stein's Method to Universality}, volume 192 of {\em Cambridge Tracts in Mathematics}.
\newblock Cambridge University Press, 2012.

\bibitem{peccati_wiener_2011}
Giovanni Peccati and Murad~S. Taqqu.
\newblock {\em Wiener Chaos: Moments, Cumulants and Diagrams}, volume~1 of {\em Bocconi \& Springer Series}.
\newblock Springer Milan, Milano, 2011.

\bibitem{SalasSokal1997}
Jes\'us Salas and Alan~D. Sokal.
\newblock Absence of phase transition for antiferromagnetic {Potts} models via the {Dobrushin} uniqueness theorem.
\newblock {\em Journal of Statistical Physics}, 86:551--579, 1997.

\bibitem{saulis_limit_1991}
L.~Saulis and V.~A. Statulevi\v{c}ius.
\newblock {\em Limit Theorems for Large Deviations}, volume~73 of {\em Mathematics and its Applications}.
\newblock Kluwer, Dordrecht, 1991.

\bibitem{shiryaev_probability-2_2019}
Albert~N. Shiryaev.
\newblock {\em Probability-2}, volume~95 of {\em Graduate Texts in Mathematics}.
\newblock Springer New York, New York, NY, 2019.

\bibitem{taqqu_convergence_1979}
Murad~S. Taqqu.
\newblock Convergence of integrated processes of arbitrary {Hermite} rank.
\newblock {\em Zeitschrift f\"ur Wahrscheinlichkeitstheorie und verwandte Gebiete}, 50(1):53--83, 1979.

\bibitem{thoma_thermodynamic_2023}
Eric Thoma.
\newblock Thermodynamic and scaling limits of the non-{Gaussian} membrane model.
\newblock {\em The Annals of Probability}, 51(2):626--664, 2023.

\bibitem{Weitz2005}
Dror Weitz.
\newblock Combinatorial criteria for uniqueness of {Gibbs} measures.
\newblock {\em Random Structures \& Algorithms}, 27(4):445--475, 2005.

\bibitem{wu_local_2025}
Wei Wu.
\newblock Local central limit theorem for gradient field models.
\newblock {\em Probability Theory and Related Fields}, 191(1):1--40, 2025.

\end{thebibliography}

\end{document}